\documentclass[
a4paper
]{amsart}
\usepackage{mathtools}
\usepackage{amssymb, amsthm, amsmath, pifont} 
\usepackage[usenames]{color}
\usepackage{paralist}
\usepackage{amsfonts}
\usepackage{enumerate}
\usepackage{pdfpages}
\usepackage[all]{xypic}
\usepackage{graphicx}
\usepackage{psfrag}
\usepackage{verbatim}
\usepackage{pstricks}
\usepackage[latin1]{inputenc}
\usepackage{mathrsfs}
\usepackage[all,knot]{xy}
\usepackage{tikz}
\usepackage{caption}
\usepackage[lofdepth,lotdepth]{subfig}
\usepackage{thm-restate}

\usepackage{thmtools}
\usepackage{thm-restate}

\usepackage{hyperref}

\usepackage{cleveref}
\usepackage
{todonotes}

\declaretheorem[name=Theorem]{thm}

\newtheorem{corollary}[thm]{Corollary}

\newtheorem{lemma}[thm]{Lemma}
\newtheorem{question}[thm]{Question}
\newtheorem{example}[thm]{Example}
\newtheorem{definition}[thm]{Definition}
\newtheorem*{intdef}{Definition}
\newtheorem*{claim*}{Claim}
\newtheorem{remark}[thm]{Remark}
\newtheorem{prop}[thm]{Proposition}
\newtheorem{claim}[thm]{Claim}
\newtheorem{obs}[thm]{Observation}

\newenvironment{Example}{\begin{example}\rm}{\end{example}}
\newenvironment{Definition}{\begin{definition}\rm}{\end{definition}}

\newenvironment{Remark}{\begin{remark}\rm}{\end{remark}}

\def\et{\;\mbox{and}\;}

\def\DTC{\mathrm{DTC}}
\DeclareMathOperator{\Ima}{Im}

\def\fd{{\omega}}

\def\for{\quad\mbox{for }}

\def\et{\quad\mbox{and}\quad}

\def\epsilon{\varepsilon}
\def\N{\mathbb{N}}
\def\Q{\mathbb{Q}}
\def\R{\mathbb{R}}
\def\Z{\mathbb{Z}}
\def\C{\mathbb{C}}

\def\D{\mathbb{D}}
\def\Re{\rm{Re}}
\def\G{\mathrm{MCG}}

\begin{document}
\title
{The Dehn twist coefficient for big and small mapping class groups}
\author{Peter Feller}
\author{Diana Hubbard}
\author{Hannah Turner}

\address{Universit\'e de Neuch\^atel, Rue Emile-Argand 11, 2000 Neuch\^atel, Switzerland}
\email{peter.feller@unine.ch}

\address{Brooklyn College CUNY, 2900 Bedford Avenue 11210 USA}
\email{diana.hubbard@brooklyn.cuny.edu}

\address{Georgia Institute of Technology}
\email{hannah.turner@math.gatech.edu}


\begin{abstract}
We study a quasimorphism, which we call the Dehn twist coefficient (DTC), from the mapping class group of a surface (with a chosen compact boundary component) that generalizes the well-studied fractional Dehn twist coefficient (FDTC) to surfaces of infinite type. Indeed, for surfaces of finite type the DTC coincides with the FDTC. We provide a characterization of the DTC as the unique homogeneous quasimorphism satisfying certain positivity conditions. This characterization is new even for the classical finite-type case and requires minimal input beyond elementary topology.
 
The FDTC has image contained in $\mathbb{Q}$. In contrast to this, we find that for some surfaces of infinite type the DTC has image all of $\mathbb{R}$. To see this we provide a new construction of maps with irrational rotation behavior for some surfaces of infinite type with a countable space of ends or even just one end. In fact, we find that the DTC is the right tool to detect irrational rotation behavior, even for surfaces without boundary.
\end{abstract}

\maketitle

\section{Introduction}

In this paper, a \emph{surface} is a smooth connected oriented $2$-manifold possibly with boundary. Let $\Sigma$
be a surface with at least one compact boundary component $\partial$. The \emph{mapping class group}
$\G(\Sigma) 
$ of $\Sigma$ is the set of isotopy classes of orientation preserving self-diffeomorphisms of $\Sigma$ that fix all boundary components pointwise.\footnote{Later in this paper we also consider the version of the mapping class group where only $\partial$ is fixed, but restrict to the standard version in the introduction for simplicity of exposition.} Denote by $T\in \G(\Sigma) 
$ the positive Dehn twist along a simple closed curve that is parallel to $\partial$.
We further make the standing assumption that $T$ is a non-trivial element in $\G(\Sigma) 
$, that is, we assume that $\G(\Sigma) 
$ is not the trivial group.


\begin{thm}\label{thm:main} There exists a unique homogeneous quasimorphism $\fd\colon \G(\Sigma)
\to\R$
such that $\fd(T)=1$ and $\fd(g)$ is non-negative for all $\partial$-right-veering $g\in\G(\Sigma)$.
\end{thm}
We call this quasimorphism $\fd$ the \emph{Dehn twist coefficient (DTC)}. We will see below in Propositions~\ref{prop:charFDTCforSurfaces} and~\ref{prop:charFDTCforSurfacestotal} that Theorem~\ref{thm:main} holds when the set of $\partial$-right-veering elements is replaced by many other sets of positive elements with respect to certain left-invariant relations or orders. 

If $\Sigma$ is of finite type (i.e.~has finitely generated fundamental group), $\G(\Sigma)
$ is the usual mapping class group and the DTC coincides with the well-studied fractional Dehn twist coefficient (FDTC). Theorem~\ref{thm:main} provides a characterization of the FDTC, which to our knowledge is new.
We discuss the FDTC (which can be informally thought of as a measure of twisting about $\partial$) in more detail in the last subsection of the introduction. 
Theorem~\ref{thm:main} also provides a natural extension of the concept of the FDTC to mapping class groups of surfaces of infinite type (i.e.~those with infinitely generated fundamental group), known as \emph{big mapping class groups}. 

The DTC has several features familiar from the FDTC for surfaces of finite type. However,
for big mapping class groups it turns out that one important feature of the FDTC, rationality of the image of $\fd$, no longer needs to hold for the DTC, which is why we drop `fractional' from the name. In fact, the image of the DTC can be all of $\R$. 

\begin{restatable}{thm}{allofR}
\label{thm:allofR}
There exist surfaces $\Sigma$ for which $\fd\colon \G(\Sigma)
\to\R$ is surjective.
\end{restatable}

See Theorem~\ref{thm:allofRprecise} for a more precise version of Theorem~\ref{thm:allofR}, in which we establish surjectivity of  $\fd$ for two infinite classes of surfaces. For the first class, the construction of the mapping classes with prescribed irrational DTC $\lambda$ relies on a homeomorphism, presumably known by Poincar\'e and explicitly constructed by Denjoy, of the circle with rotation number $\lambda+\Z$ that fixes a Cantor set $C$. This yields a big mapping class by embedding the circle into a surface, extending the homeomorphism to the surface and then removing $C$. This construction is known to experts in big mapping class groups (see e.g.~\cite[Proof of Theorem~5.1]{CalegariChen22}), and these mapping classes are referred to as \emph{irrational rotations}
(e.g.~in \cite[Section~4]{BFT}). For the second class, the result depends on a novel explicit construction of mapping classes with prescribed irrational DTC $\lambda$ which we call the \emph{wagon wheel maps of rotation $\lambda$}. These new mapping classes are of independent interest. In fact, we will see that the wagon wheel maps and their capped off versions are relevant examples in the context of a recent conjecture by Bestvina, Fanoni, and Tao~\cite[Conjecture B]{BFT}.

Achieving rational DTC is simpler by comparison. In fact, we show that for every surface of infinite type, every rational number lies in the image of DTC.

\begin{prop}\label{prop:allofQ}
Let $\Sigma$ be a surface of infinite type with a chosen compact boundary component $\partial$. Then the image of $\omega: \G(\Sigma)
\to \mathbb{R}$ contains all rational numbers.
\end{prop}

In contrast, recall that for a surface of finite type the possible values of the FDTC mod 1 are finite, constrained by the topology of the surface \cite{malyutin2005twist, ItoKawamuro_OpenBookFoliations}. 
The proof of Proposition~\ref{prop:allofQ} uses a family of braids with the desired `twisting' to construct related maps on surfaces of infinite type.
 
Given Theorem~\ref{thm:allofR} and Proposition~\ref{prop:allofQ}, we wonder whether the DTC is always surjective for surfaces of infinite type.

\begin{question} Let $\Sigma$ be any surface of infinite type with a chosen compact boundary component. Is $\omega: \G(\Sigma)
\to \mathbb{R}$ surjective?
\end{question}


We understand Theorem~\ref{thm:main} and Theorem~\ref{thm:allofR} as showing that it is impossible to find a sensible extension of the FDTC to big mapping class groups that preserves the property of rationality.
We note the existing proof that the FDTC is rational for surfaces of finite type relies on the Nielsen-Thurston classification for mapping classes. Theorem~\ref{thm:allofR} may therefore be seen as an indication that an analogue of a Nielsen-Thurston type classification of maps for surfaces of infinite type ought to be quite complicated. We discuss in more detail below how the irrationality of the DTC allows one to measure irrational rotation behavior. 


\subsection*{Application: detecting irrational rotation behavior}

Bestvina-Fanoni-Tao consider a subset of big mapping classes called tame and speculate that tame mapping classes are made up of particular kinds of pieces \cite[Conjecture~B]{BFT}. One such piece, a mapping class with `irrational rotation behavior', is discussed by example: the irrational rotation of the sphere with a Cantor set removed.
In Section~\ref{sec:tameness}, we use the DTC to propose a definition of irrational rotation behavior. {In particular, we define a rotation number (valued in $\R/\Z$) for mapping classes that fix some isolated planar end, or some boundary, and for diffeomorphisms which fix a point.}
 We say such a mapping class or diffeomorphism has \emph{irrational rotation behavior} if the rotation number is irrational.  We show that wagon wheel maps---the new example of a big mapping class in this paper---are 
examples of irrational rotation behavior as we define it; see Proposition~\ref{prop:irrbehforwagonwheels}. 

We also show that our proposed definition of irrational rotation behavior is such that the irrational rotations on the sphere without a Cantor set (and a natural choice of fixed point) and its variants are also all examples with irrational rotation behavior; see Proposition \ref{prop:irrbehforirrarots}. In fact, as one would expect, the irrational rotations constructed from a homeomorphism of the circle with rotation number $r\in \R/\Z\setminus \Q/\Z$ have rotation number $r$.

It is our understanding that there has so far not been a need to formalize `irrational rotation behavior' since no examples beyond the irrational rotations appear in the literature.
Our proof of Theorem~\ref{thm:allofR} changes this since we construct wagon wheel maps with irrational rotation behavior---in fact with any irrational rotation number---for, among others,
the surface with one end, no boundary and infinite genus, and the planar surface with no boundary and countable space of ends with one accumulation point (aka $\R^2\setminus \Z^2$).

\subsection*{Relation to the fractional Dehn twist coefficient and left-orders}

There is a long history of the study of the fractional Dehn twist coefficient (FDTC) associated with a choice of a compact boundary component $\partial$ of a surface $\Sigma$ of finite type. Gabai and Oertel used what they called the \emph{degeneracy slope} of a fibered knot to construct essential laminations in 3-manifolds~\cite{Gabai_EssentialLaminations3Manifolds}. Dividing 1 by the degeneracy slope yields the fractional Dehn twist coefficient as studied by Honda-Kazez-Matic \cite{HondaKazezMatic_RightVeeringI}, where it is used to study contact structures on 3-manifolds, and, in particular, characterize tightness via the notion of right-veering. Work in this vein has continued to be fruitful \cite{HondaKazezMatic_RightVeeringII,ColinHonda_ReebVFs,KazezRoberts_FDTC, ItoKawamuro_OpenBookFoliations}.

In another direction, the fractional Dehn twist coefficient of a braid, also called the \emph{twist number}, has been used to relate properties of a braid to its closure; see \cite{malyutin2005twist, ItoKawamuro_OpenBookFoliations,feller_hubbard,Feller22_sBiforFDTC} among others. Malyutin also characterizes the FDTC using Dehornoy's left-order on the braid group (see Propostion~\ref{prop:CharPropForFDTC}). It is this characterization that inspires our definition of the DTC. In analogy to the Dehornoy order, in Section \ref{sec:totalorders} we 
produce total left-orders (and circular orders) on big mapping class groups and discuss the connection of these orders to the DTC.

Our construction of the DTC $\fd\colon \G(\Sigma)\to\R$, while informed by prior ideas and techniques, is completely self-contained. In the case of surfaces of finite type, it recovers the FDTC.
\begin{corollary}\label{cor:DTC=FDTC}
If $\Sigma$ is of finite type, the \emph{DTC} $\fd\colon \G(\Sigma)\to \R$ equals the \emph{FDTC}.
\end{corollary}
This is a rather immediate consequence of the uniqueness in Theorem~\ref{thm:main} and the well-known properties of the FDTC, so we include the proof here.
\begin{proof}[Proof of Corollary~\ref{cor:DTC=FDTC}] The FDTC is a homogeneous function from $\G(\Sigma)$ to $\R$ that 
is $1$ on $T$ 
  and non-negative on $\partial$-right-veering mapping classes~\cite{HondaKazezMatic_RightVeeringI}.
Furthermore, it is a quasimorphism~\cite{ItoKawamuro_OpenBookFoliations} (see also~\cite{malyutin2005twist} for the case of braid groups).
Hence, by Theorem~\ref{thm:main}, it must be equal to the DTC $\fd\colon\G(\Sigma)\to\R$.
\end{proof}

Theorem~\ref{thm:main} not only immediately yields that the DTC is equal to the FDTC, but as we previously noted, it gives a new characterization of the FDTC. In particular, our proof of the existence and uniqueness allows for a construction of the FDTC that does not require a choice of hyperbolic metric or the use of the Nielsen-Thurston classification for mapping class groups. As mentioned previously, the DTC shares many properties with the FDTC.
For example, $\fd$ has defect at most $1$
; see Proposition~\ref{prop:charFDTCforSurfaces}. Also, $\fd(g)>0$ implies that $g$ is $\partial$-right veering (see Corollary~\ref{cor:fd(g)>0=>gisrightveering}), generalizing the corresponding statement for the FDTC~\cite[Proposition 3.1 and~3.2]{HondaKazezMatic_RightVeeringI}.  Due to our new characterization, this last property and others get a topological proof free of geometry or the Nielsen-Thurston classification. In contrast to this, given our Theorem~\ref{thm:allofR}, the use of the Nielsen-Thurston classification in establishing the rationality of the FDTC (see~\cite{HondaKazezMatic_RightVeeringI, ItoKawamuro_OpenBookFoliations}) might be essential.

\subsection*{Organization} We describe sufficient conditions for a group with a relation to admit a unique homogeneous quasimorphism to $\R$ in Section~\ref{sec:qmviafloors}. Using this algebraic setup, we prove Theorem \ref{thm:main} in Section~\ref{sec:MCG}, leaving some technical details to Section~\ref{sec:annular}. These technical details amount to an explicit description of the lift of a self-homeomorphism of a surface to its annular cover. 
In Section~\ref{sec:irrationality}, we turn our attention towards realizing irrational numbers in the image of the DTC on big mapping class groups. The key point is explicit constructions of the wagon wheel maps and irrational rotations, which we employ to prove the hard part of Theorem~\ref{thm:allofR}: the realization of irrational numbers. The proof of the portion of Theorem~\ref{thm:allofR} about achieving all rational numbers is left to Section~\ref{section:rationals}. In Section~\ref{sec:tameness}, we discuss the wagon wheel maps in the context of a conjecture from~\cite{BFT}, we propose a definition of irrational rotation behavior, and check that wagon wheel maps have irrational rotation behavior. In Section~\ref{sec:totalorders}, we end the paper by discussing total orders on big mapping class groups and provide a variant of Theorem~\ref{thm:main} with the positive cone of certain total orders in place of the set of $\partial$-right veering elements.
\subsection*{Acknowledgements} This project originated from discussions at ICERM during the semester program  `Braids.' We thank ICERM for making the collaboration possible and supporting our work. PF was supported by the SNSF
grant 181199. DH was supported by NSF grant 2213451 and PSC-CUNY grants 64412-00 52 and 65473-00 53.
HT was supported by NSF grant 1745583 (Georgia Tech), and DMS-1929284 (ICERM). We would like to thank the anonymous referee for many detailed comments and suggestions which improved the paper. We also acknowledge helpful conversations with Mladen Bestvina, Federica Fanoni, Florian Stecker, and Nicholas Vlamis.

 \section{Existence and uniqueness of homogeneous quasimorphisms that are positive on positive elements}\label{sec:qmviafloors}
In this section, we describe an algebraic setup that guarantees the existence and uniqueness of homogeneous quasimorphisms with certain positivity properties. We understand this as a generalization of the characterization of the FDTC for the braid group described in the next subsection. We feel justified in providing this somewhat abstract setup since this allows us to provide a uniqueness and existence result that provides the natural extension of the FDTC to all big mapping class groups with rather little explicit topological work; see Section~\ref{sec:MCG}. 

We recall that a \emph{quasimorphism} on a group $G$ is a function $f\colon G \to\R$ such that
$\sup_{a,b\in G}|f(ab)-f(a)-f(b)|<\infty$,
where $D_f\coloneqq\sup_{a,b\in G}|f(ab)-f(a)-f(b)|$ is known as the \emph{defect} of $f$.
A function $f\colon G\to\R$ is said to be \emph{homogeneous}
if $f(g^k)=kf(g)$ for all $g\in G$ and $k\in\Z$. 
\subsection*{Characterizing the FDTC for the braid group using Dehornoy positivity} One may characterize the FDTC on the braid group $B_n$ as follows.
\begin{prop}[{\cite[8.1~Theorem]{malyutin2005twist}}]\label{prop:CharPropForFDTC}For every $n\geq2$,
there exists exactly one homogenous quasimorphism
$\fd\colon B_n \to \R$ such that
\begin{enumerate}[$($i\,$)$]
  \item $\fd\left(\Delta^2\right)=1$.
  \item\label{item:pos}$\fd(\beta)\geq 0$ for all $\beta\in B_n$ with $\beta \succ_\textrm{Deh} 1$.
\end{enumerate}
\end{prop}

Observe that~\eqref{item:pos} depends on the notion of Dehornoy positivity, also called $\sigma$-positivity; see e.g.~\cite{DDRW-OrderingBraids}. However, it turns out that one may use any number of notions of positivity 
and the statement of Proposition~\ref{prop:CharPropForFDTC} 
holds. In fact, statements analogous to Proposition~\ref{prop:CharPropForFDTC} hold for (big and small) mapping class groups (see Section~\ref{sec:MCG}), and one may phrase this rather generally for groups with a certain order or relation on them (see Proposition~\ref{prop:uniqunessviapositivityongroups} below).

\subsection*{Characterizing quasimorphisms on central extensions of the infinite cyclic group}
The following makes precise the claim from above that Proposition~\ref{prop:CharPropForFDTC} may be phrased rather generally. {We include the general statement of Proposition \ref{prop:uniqunessviapositivityongroups} as it may be of broader use. However, in our applications, we will take $G$ to be a mapping class group, $T$ to be a Dehn twist around a boundary component, and in this setting we can take $K=0$. }

\begin{prop}\label{prop:uniqunessviapositivityongroups} Let $G$ be a group, $T$ be an element in its center, and $\preceq$ a left-invariant relation on $G$ such that there exists an integer $K$ with
\begin{align}\label{(T)}
\text{for all $\beta\in G$, $1\preceq \beta$ or $\beta\preceq T^K$.}
\end{align}
There exists at most one homogeneous quasimorphism $f\colon G\to \R$ such that
\begin{enumerate}[$($i\,$)$]
\item\label{item:i} $f(T)=1$
\item\label{item:ii} $f\left(\left\{\beta\in G\mid \beta\succeq 1\right\}\right)\subseteq\R$ is bounded below.
\end{enumerate}

If, furthermore, the relation $\preceq$ is transitive and we have the following \emph{Archimedean} property: for each $\beta\in G$ there exists an integer $k_\beta$ such that $\beta\preceq T^{k_\beta}$ and $\beta\not\preceq T^l$ for $l\leq -k_\beta$, then there does exist a homogeneous quasimorphism $f\colon G\to \R$ such that~\eqref{item:i} and~\eqref{item:ii}. In this case, the defect $D_f$ of $f$ is bounded by $K+1$ and $f\left(\left\{\beta\in G\mid \beta\succeq 1\right\}\right)\subseteq[0,\infty)$.
\end{prop}

 Proposition~\ref{prop:uniqunessviapositivityongroups} immediately recovers Proposition~\ref{prop:CharPropForFDTC} by noting that the Dehornoy order is a left-invariant total order. As such it in particular satisfies~\eqref{(T)}, which can be understood as a coarse version of totality.
\begin{proof}[Proof of Proposition~\ref{prop:CharPropForFDTC}]
Apply Proposition~\ref{prop:uniqunessviapositivityongroups} to $G=B_n$, $T=\Delta^2$, and $\preceq=\preceq_\textrm{Deh}$.
\end{proof}
\begin{proof}[Proof of Proposition~\ref{prop:uniqunessviapositivityongroups}]
We first prove the uniqueness statement. Let $f$ and $g$ be homogeneous quasimorphisms such that  
\[f(T)=g(T)=1\et f\left(\left\{\beta\in G\mid \beta\succeq 1\right\}\right),g\left(\left\{\beta\in G\mid \beta\succeq 1\right\}\right)\subseteq [C,\infty)\] for some $C\in\R$.
Assume towards a contradiction that $f\neq g$, i.e.~there exists an $\alpha\in G$ with $g(\alpha)-f(\alpha)>0$. Pick $m\in\N$ such that $g(\alpha^m)-f(\alpha^m)>1$ and pick $k\in\Z$ such that
\[k+g(\alpha^m)\overset{\text{\eqref{item:i}}}{=}g(T^k)+g(\alpha^m)=g(T^k\alpha^m)>0>f(T^k\alpha^m)=f(T^k)+f(\alpha^m)\overset{\text{\eqref{item:i}}}{=}k+f(\alpha^m),\]
where the second and second to last equality uses that homogeneous quasimorphisms restrict to group homomorphisms on Abelian subgroups.
We set $\beta\coloneqq T^k\alpha^m$.
Since $f(\beta)<0$, there exists an $n_0\in \N$ such that $f(\beta^{n_0n} T^{n})=n(n_0f(\beta)+1)<C$ for $n\in \N$. Hence, $1\not\preceq \beta^{n_0n} T^{n}$ and, thus, by~\eqref{(T)}, $\beta^{n_0n} T^{n}\preceq T^{K}$ for $n\in\N$.
This leads to the following contradiction. For all $n\in\N$, we have
\[C\leq g\left(T^{K-n}\beta^{-n_0n}\right)=g\left(T^{K-n}\right)+g\left(\beta^{-n_0n}\right)\leq K-n-n_0ng(\beta)\leq K-n.
\]

We now turn to the existence statement. For this we show the following claim.
\begin{claim}\label{claim}$\lfloor\beta\rfloor\coloneqq \max\{{k\in\Z}\mid T^k\preceq\beta\}$, called the \emph{floor of $(G,T,\preceq)$}, defines a function from $G$ to $\Z$ that satisfies
\[0\leq\lfloor \alpha\beta\rfloor-\lfloor \alpha\rfloor-\lfloor\beta\rfloor\leq K+1\quad\text{ for all }\alpha,\beta\in G.\]
\end{claim}
\begin{proof}[Proof of Claim~\ref{claim}]
To see that the floor is well-defined, we note that $\{{k\in\Z}\mid T^k\preceq\beta\}$ is a nonempty and bounded-above subset of $\Z$ (since, by the Archimedean property, the set contains $-k_{\beta^{-1}}$ and contains no $k\geq k_{\beta^{-1}}$) and hence it has a maximum in $\Z$.

Next we show $0\leq\lfloor \alpha\beta\rfloor-\lfloor \alpha\rfloor-\lfloor\beta\rfloor$ for all $\alpha,\beta\in G$.
Since $1\preceq\alpha T^{-\lfloor\alpha\rfloor}$ (by the definition of $\lfloor\alpha\rfloor$) and $\alpha T^{-\lfloor\alpha\rfloor}\preceq\alpha T^{-\lfloor\alpha\rfloor}\beta T^{-\lfloor\beta\rfloor}$ (by the definition of $\lfloor\beta\rfloor$ and left-invariance), transitivity yields
$1\preceq\alpha T^{-\lfloor\alpha\rfloor}\beta T^{-\lfloor\beta\rfloor}=\alpha\beta T^{-\lfloor\beta\rfloor-\lfloor\alpha\rfloor}$. Hence, we have $\lfloor\alpha\beta\rfloor\geq \lfloor\beta\rfloor+\lfloor\alpha\rfloor$.

Finally, we show $\lfloor \alpha\beta\rfloor-\lfloor \alpha\rfloor-\lfloor\beta\rfloor\leq K+1$ for all $\alpha,\beta\in G$.
Note that $1\not\preceq \alpha T^{-(\lfloor \alpha\rfloor+1)}$, hence
$\left(\alpha T^{-(\lfloor \alpha\rfloor+1)}\right)\preceq T^{K}$. Rewriting this as
$\left(T^{-(\lfloor \alpha\rfloor+1+K)}\right)\alpha\preceq 1$ (by left-invariance) and using
$1\preceq\alpha\beta T^{-\lfloor\alpha\beta\rfloor}$ we have (employing transitivity and left-invariance) that
\[\left(T^{-(\lfloor \alpha\rfloor+1+K)+\lfloor\alpha\beta\rfloor}\right)=\left(T^{-(\lfloor \alpha\rfloor+1+K)}T^{\lfloor\alpha\beta\rfloor}\right)\preceq\beta.\] Hence, $\lfloor\beta\rfloor\geq -(\lfloor \alpha\rfloor+1+K)+\lfloor\alpha\beta\rfloor$ as desired.
\end{proof}

Claim~\ref{claim} can be rephrased as the observation that the floor defines a superadditive quasimorphism of defect $K+1$.
Using the claim, we find a homogeneous quasimorphism $f\colon G\to \R$ of defect $D_f\leq K+1$ as the homogenization of the floor. For this, recall that every quasimorphism to $\R$ has a well-defined homogenization quasimorphism (by Fekete's Lemma). In general, the defect of the homogenization is at most twice the defect of the original quasimorphism; however, for superadditive quasimorphisms, the homogenization has defect bounded by the defect of the quasimorphism. In other words,
$f(\beta)\coloneqq \lim_{n\to\infty}\frac{\lfloor\beta^n\rfloor}{n}$ is a homogeneous quasimorphism with defect at most $K+1$. In fact, since $\lfloor \cdot \rfloor$ is superadditive, by Fekete's lemma not only does the limit exist, it is in fact equal to the supremum,
i.e.~$f(\beta)=\lim_{n\to\infty}\frac{\lfloor\beta^n\rfloor}{n}=\sup_{n\in\N}\left\{\frac{\lfloor\beta^n\rfloor}{n}\right\}$. 

It remains to check that $f$ satisfies~\eqref{item:i} and~\eqref{item:ii}.
To see~\eqref{item:i}, we let $k_1\in\Z$ be such that $1\preceq T^{k_1}$ and $1\not\preceq T^l$ for $l\leq -k_1$. Then $1\preceq T^{-(n-k_1)}T^{n}$ and $1\not\preceq T^{-k}T^n$, for $n-k\leq -k_1$. Hence, $n-k_1\leq\lfloor T^n\rfloor\leq n+k_1-1$ for all $n\in\N$, which implies $f(T)\coloneqq \lim_{n\in\N}\frac{\lfloor T^n\rfloor}{n}=1$.
Finally, to see~\eqref{item:ii}, we note that, $\lfloor \beta^n\rfloor\geq 0$ if $1\preceq \beta$, since, if $1\preceq \beta$, then $1\preceq \beta^n$ for $n\in \N$ (by transitivity and left-invariance). Hence, $f(\beta)\geq 0$ for all $\beta$ such that $1\preceq \beta$.
\end{proof}

In Section \ref{sec:MCG} we apply Proposition~\ref{prop:uniqunessviapositivityongroups} with $G$ a fixed mapping class group and $T$ a fixed positive Dehn twist.
All left-invariant relations we will consider on $G$ will be comparable, in the sense described next. As a consequence, the unique homogeneous quasimorphism from Proposition~\ref{prop:uniqunessviapositivityongroups} will always be the same: the DTC.

For a group $G$ and a fixed element $T$ in its center, we call two left-invariant relations $\preceq_1$ and $\preceq_2$ on $G$ \emph{comparable} if there exists a constant $d\in\N$ such that $\alpha\preceq_1\beta\Rightarrow\alpha\preceq_2 T^d\beta$ and
$\alpha\preceq_2\beta\Rightarrow\alpha\preceq_1 T^d\beta$ for all $\alpha$ and $\beta$ in $G$.
\begin{Remark}\label{rmk:setvsinvR}
A left-invariant relation $\preceq$ on $G$ is uniquely described by its set of positive elements $P\coloneqq\left\{g\in G\mid 1 \preceq g\right\}$, and, vice-versa, any subset of $G$ is the set of positive elements of a unique left-invariant relation $\preceq$, namely $\alpha\preceq\beta$ if and only if $\alpha^{-1}\beta\in P$. Properties of left-invariant relations $\preceq$ are easily restated in terms of $P$. For example, $\preceq$ is transitive if and only if $P$ is multiplicatively closed. 

Note that $\preceq_1$ and $\preceq_2$ on $G$ are comparable if and only if there exists a constant $d\in\N$ such that $T^d P_1\subseteq P_2$ and $T^d P_2\subseteq P_1$, where $P_1$ and $P_2$ denote the set of positive elements of $\preceq_1$ and $\preceq_2$, respectively.
Also note that, if $\preceq_1$ and $\preceq_2$ are comparable, then the assumption~\eqref{(T)} 
in Proposition~\ref{prop:uniqunessviapositivityongroups} holds for $\preceq=\preceq_1$ if and only if it holds (in general with different $K$) for $\preceq=\preceq_2$ and a homogeneous quasimorphism satisfies~\eqref{item:i} and~\eqref{item:ii} for $\preceq=\preceq_1$ if and only if it satisfies them for $\preceq=\preceq_2$. { Hence, any comparable $\preceq_1$ and $\preceq_2$ satisfying assumption (1) will define the same quasimorphism.}
\end{Remark}

\section[Characterization (and Definition) of the DTC]{Characterization (and Definition) of 
the Dehn twist coefficient}
\label{sec:MCG}

Let $\Sigma$ be a surface (a smooth connected oriented $2$-manifold) with non-empty boundary $\partial \Sigma$ and let $\partial$ be a compact boundary component. From this point on, we will consider two possible flavors of the mapping class group of $\Sigma$.
By $\G (\Sigma,\partial)$ we mean the set of isotopy classes of orientation preserving self-diffeomorphisms of $\Sigma$ that fix the special boundary component $\partial$ pointwise, and by $\G (\Sigma, \partial \Sigma)$ we mean the isotopy classes of orientation preserving self-diffeomorphisms of $\Sigma$ that fix all boundary components pointwise. When results apply to both flavors of the mapping class group we will write simply $\G(\Sigma)$. 
Note that in the introduction we defined $\G(\Sigma)$ to mean $\G (\Sigma,\partial)$ for simplicity of exposition. However, all statements in the introduction do in fact apply to both flavors.
 
The classical mapping class group setup corresponds to the case of the surface $\Sigma$ being of finite type, i.e.~$\Sigma$ has finitely generated fundamental group.
For example, the braid group on $n$-strands arises as $\G(\Sigma,\partial)$, where $\Sigma$ is the $n$-punctured closed disk, i.e.~the result of removing $n$ points from the interior of $\mathbb{D}$.

Denote by $T\in \G(\Sigma)$ the positive Dehn twist along a simple closed curve that is parallel to $\partial$. 
Recall that we made the standing assumption that $T$ is a non-trivial element in $\G(\Sigma)$. This excludes the cases of $\Sigma$ being a closed disk or a once punctured closed disk and the case of $\G(\Sigma)$ being $\G(S^1\times[0,1],S^1\times\{0\})$.
We define pairwise comparable left-invariant relations by describing their sets of positive elements $P\subseteq \G(\Sigma)$ as discussed in Remark~\ref{rmk:setvsinvR}.

We call a continuous map from $\gamma\colon [0,1]\to\Sigma$ an arc in $\Sigma$. 
Fix a basepoint $x_0\in\partial$ and denote by $\mathfrak{I}$ the set of homotopy classes
of 
arcs
$I$ in $\Sigma$ that start at $x_0$, end in $\partial\Sigma$, and that are not boundary parallel, where homotopies are restricted on the boundary as follows. The starting point is fixed and the endpoints are fixed whenever the endpoint lies on a boundary component that remains fixed in the definition of $\G(\Sigma)$. That is, if $\G(\Sigma)=\G(\Sigma,\partial \Sigma)$ or if $\G(\Sigma)=\G(\Sigma,\partial)$ and the endpoint is on $\partial$ the homotopies are rel both endpoints, otherwise they fix the starting point and always map the endpoint into $\partial \Sigma$. 
Given this setup, it is clear that $\G(\Sigma)$ acts from the left on $\mathfrak{I}$ by $g(I)=[\phi(\gamma)]$, where $g=[\phi]\in \G(\Sigma)$ and $I=[\gamma]$. 

Let $I\neq J\in \mathfrak{I}$. Denote by $\widetilde{\Sigma}$ the universal cover of $\Sigma
$, choose a basepoint $\widetilde{x_0}$ that is the lift of ${x_0}$, and let $\widetilde{I}$ and $\widetilde{J}$ be the homotopy classes of lifts of $I$ and $J$, respectively, with initial point $\widetilde{x_0}$. Let $\gamma_{\widetilde{I}}$ and $\gamma_{\widetilde{J}}$ denote representatives of $\widetilde{I}$ and $\widetilde{J}$ that have disjoint image away from $\widetilde{x_0}$. The latter is possible since the interior of $\widetilde{\Sigma}$, denoted by $D$, is an open disk and since the endpoints of any two representatives of $\widetilde{I}$ and $\widetilde{J}$ are different due to $I\neq J$. Hence, $D\setminus \Ima(\gamma_{\widetilde{I}})$ has two connected components, $L$ (the one to the left of $\Ima(\gamma_{\widetilde{I}})$) and $R$ (the one to the right of $\Ima(\gamma_{\widetilde{I}})$), and $\Ima((\gamma_{\widetilde{J}})\cap D$ is either a subset of $L$ or of $R$.

\begin{definition}\label{def:totheright} For $I\neq J\in \mathfrak{I}$, we say $J$ is \emph{strictly to the right} of $I$, denoted by $I<J$ if $\Ima(\gamma_{\widetilde{J}})\cap D\subset R$.

For $I$ and $J$ in $\mathfrak{I}$, we say $J$ is \emph{to the right} of $I$, denoted by $I\leq J$ if
$I=J$ or $I\neq J$ and $I<J$.

\end{definition}
\begin{lemma}\label{lem:<=istotal}
The relation $\leq$ on $I\in\mathfrak{I}$ is a total order that is left-invariant under the action of $G$ on $\mathfrak{I}$.
\end{lemma}\begin{proof}
Left-invariance is clear from the definition and the fact that orientation preserving diffeomorphism lift to orientation preserving diffeomorphisms of the universal cover.
 
It remains to show (1) reflexivity, (2) transitivity, (3) anti-symmetry and (4) totality.
(1) Reflexivity is clear. 
(2)
For transitivity we show that if $I\leq J$ and $J\leq K$ we have $I\leq K$ for all $I, J, K\in \mathfrak{I}$. We assume $I$, $J$, and $K$ are pairwise different, since there is nothing to show otherwise. 
$I\leq J$ means (by definition) that $D\cap\Ima(\gamma_{\widetilde{J}})$ is in the right component of $D\setminus\Ima(\gamma_{\widetilde{I}})$, while $J\leq K$ means that $D\cap \Ima(\gamma_{\widetilde{K}})$ is in the {right} component of $D\setminus\Ima(\gamma_{\widetilde{J}})$. Hence, $D\cap\Ima(\gamma_{\widetilde{K}})$ is in the right component of $D\setminus\Ima(\gamma_{\widetilde{I}})$, meaning $I\leq K$, as desired.
(3) For anti-symmetry we show that if $I\leq J$ and $J\leq I$ we must have $I=J$ for $I,J\in\mathfrak{I}$. Since $I\leq J$ we have either $I=J$ or $D\cap \Ima(\gamma_{\widetilde{J}})$ is in the right component of $D\setminus\Ima(\gamma_{\widetilde{I}})$. Similarly, $J\leq I$ implies either $I=J$ or $D\cap \Ima(\gamma_{\widetilde{J}})$ is in the left component of $D\setminus\Ima(\gamma_{\widetilde{I}})$. Since $\Ima(\gamma_{\widetilde{J}})$ cannot lie both in the left and right component of $D\setminus\Ima(\gamma_{\widetilde{I}})$ we must have $I=J$.
(4) Totality is clear from the definition.
\end{proof}


We write $P_I\subseteq \G(\Sigma)$ for the set of the elements $g\in \G(\Sigma)$ such that $I\leq g(I)$ and define
\[P_\textrm{R-veering}\coloneqq \bigcap_{I\in\mathfrak{I}}P_I\]
and call $P_\textrm{R-veering}$ the set of mapping classes that are \emph{$\partial$-right-veering}. 
We further write $P_\Gamma=\bigcap_{I\in\Gamma}P_I$ for $\Gamma\subseteq \mathfrak{I}$.


\begin{prop}\label{prop:charFDTCforSurfaces}
For every non-empty subset $\Gamma$ of $\mathfrak{I}$, 
there exists a unique homogeneous quasimorphism $\fd\colon \G(\Sigma)\to \R$  with $\fd(T)=1$ such that
$\fd(P_{\Gamma})$ is bounded below. This unique quasimorphism has defect at most $1$, and,
in fact, $\fd$ is the same for all non-empty subsets $\Gamma$ of $\mathfrak{I}$ and it satisfies $\fd(P_{\Gamma})\subseteq [0,\infty)$ for all of them. 
\end{prop}

{We delay the proof of Proposition \ref{prop:charFDTCforSurfaces} to the end of this section, at which point the result will follow from Proposition \ref{prop:uniqunessviapositivityongroups}, with some technical details  pushed to Section \ref{sec:annular}. }

\begin{definition}
We call $\fd\colon \G(\Sigma)\to \R$ from Proposition~\ref{prop:charFDTCforSurfaces} the \emph{Dehn twist coefficient (DTC)} of $\G(\Sigma)$. 
\end{definition}

We note that by setting  $\Gamma=\mathfrak{I}$, Proposition~\ref{prop:charFDTCforSurfaces} in particular proves Theorem~\ref{thm:main}.
\begin{proof}[Proof of Theorem~\ref{thm:main}]
By definition, $g\in \G(\Sigma)$ is $\partial$-right-veering if $g(I)\geq I$ for all $I\in \mathfrak{I}$. Hence, considering Proposition~\ref{prop:charFDTCforSurfaces} with $\Gamma=\mathfrak{I}$ completes the proof.
\end{proof}

Before we discuss the proof of Proposition~\ref{prop:charFDTCforSurfaces}, we note the following consequences that show that many properties of the FDTC generalize easily to the DTC. We call $g\in G$ \emph{strictly $\partial$-right-veering} if $g(I)>I$ for all $I\in\mathfrak{I}$.
\begin{corollary}\label{cor:fd(g)>0=>gisrightveering}
If $g\in\G(\Sigma)$ satisfies $\fd(g)>0$, then $g$ is strictly $\partial$-right-veering.
\end{corollary}
\begin{proof} Let $g\in\G(\Sigma)$ be not strictly $\partial$-right-veering, i.e.~there exists an $I\in\mathfrak{I}$ such that $g(I)\leq I$. Hence, $g^{-1}(I)\geq I$ and thus, by Proposition~\ref{prop:charFDTCforSurfaces} applied to $\Gamma=\{I\}$, we have $-\fd(g)=\fd(g^{-1})\geq 0$ as desired.
\end{proof}


Before we get to the next consequences, we provide an explicit formula as a limit and as a supremum for the DTC, combining elements of the proofs and definitions in this section as well as the construction in the proof of Proposition~\ref{prop:uniqunessviapositivityongroups}. The proof of Proposition \ref{prop:charFDTCforSurfaces} at the end of this section justifies that this formula is correct and well-defined, but we highlight it here for the reader's convenience.

\begin{Remark}\label{rem:DTC=lim}
Fix any $I\in \mathfrak{I}$. For any $g \in \G(\Sigma)$, declare that $1\preceq g$ if $g(I)$ is to the right of $I$. Use this definition of $\preceq$ to define a floor on $\G(\Sigma)$, namely $\lfloor g \rfloor\coloneqq \max\{{k\in\Z}\mid T^k\preceq g\}$. The Dehn twist coefficient of a fixed $f \in \G(\Sigma)$ is
\[\fd(f) = \lim_{n\to\infty}\frac{\lfloor f^n\rfloor}{n}=\sup_{n\in\N}\left\{\frac{\lfloor f^n\rfloor}{n}\right\}\coloneqq\sup_{n\in\N}\left\{\frac{k}{n}\mid k\in\Z, n\in\N\colon T^{k}\preceq f^{n}\right\}.\]
When we calculate the Dehn twist coefficients of a specific map later in this paper, we will use this approach. Namely, we will choose a convenient $I$ and use it to define a floor that we then calculate for all positive iterates of the map. 

The reader should compare this to~\cite[Theorem~7.5]{malyutin2005twist} (in the setting of braids) and to~\cite[Theorem~3]{ClayGhaswala2022cofinal} (in the setting of surfaces of finite type) which, while being different results from ours, also characterize the fractional Dehn twist coefficient in terms of the homogenization of floors.
\end{Remark}
In fact, not only is the limit of $\frac{\lfloor f^n\rfloor}{n}$ equal to the supremum, it converges quickly to $\fd(f)$. Concretely, and in analogy to what is known for the FDTC in case of surfaces of finite type (see~\cite[Theorem~4.14]{ItoKawamuro_OpenBookFoliations}), one finds the following.

\begin{corollary}\label{cor:boundsondtcintermsoffloor}
For all $f\in\G(\Sigma)$ and $n\in\N$, we have \[\frac{\lfloor f^n\rfloor}{n}\leq \fd(f)\leq\frac{\lfloor f^{n}\rfloor}{n}+\frac{1}{n}.\]
\end{corollary}
\begin{proof}
 For all $n,m\in\N$,
we have $m\lfloor f^n\rfloor\leq \lfloor f^{nm}\rfloor\leq m\lfloor f^{n}\rfloor+m-1$, where in the first inequality we applied superadditivity and in the second we applied that
$\lfloor fg\rfloor\leq \lfloor f\rfloor+\lfloor g\rfloor+1$ holds for all $f,g\in\G(\Sigma)$ (see Claim~\ref{claim} {applied with $K=0$, which will be justified in the proof of Proposition \ref{prop:charFDTCforSurfaces}}). Dividing by $nm$, we find
\begin{equation}\label{eq:fdexplicitly}
\frac{\lfloor f^n\rfloor}{n}\leq \frac{\lfloor f^{nm}\rfloor}{nm}\leq \frac{\lfloor f^{n}\rfloor}{n}+\frac{m-1}{nm}<\frac{\lfloor f^{n}\rfloor}{n}+\frac{1}{n}.
\end{equation}
The first inequality of~\eqref{eq:fdexplicitly} implies
that $\fd(f)=\sup_{m\in\N}\left\{\frac{\lfloor f^{nm}\rfloor}{nm}\right\}$ for all $n\in\N$. Hence, taking the supremum over $m$ in~\eqref{eq:fdexplicitly} yields the desired result.
\end{proof}

As an application of Remark~\ref{rem:DTC=lim}, we also note the following corollary. Recall that a useful tool for computation of the FDTC for braids is that inserting positive braid generators into a braid word can only increase the FDTC~\cite[Lemma~5.2]{malyutin2005twist}. This is an instance of the following statement: inserting $\partial$-right-veering elements into a composition of mapping classes can only increase the DTC.

\begin{corollary}\label{cor:insertrvincreasesdtc} If $r\in\G(\Sigma)$ is $\partial$-right veering, then $\fd(grh)\geq\fd(gh)$ for all $g,h\in\G(\Sigma)$.
\end{corollary}
\begin{proof} By conjugation invariance, $\fd(grh)=\fd(rhg)$ and $\fd(hg) = \fd(gh)$, thus we can instead show that $\fd(rhg) \geq \fd(hg)$.
Recall from Remark~\ref{rem:DTC=lim} that $\fd(hg)$ equals the supremum over the set of $ \frac{k}{n}$ for which $k\in\Z$ and $n\in\N$ are such that $(hg)^nT^{-k}(I)\geq I$ and $I$ is any element in $\mathfrak{I}$. In addition, $\fd(rhg)$ equals the same supremum, but with $rhg$ in place of $hg$. { Since $r$ is $\partial$-right-veering we have that $rhg(J)\geq hg(J)$ for all $J\in\mathfrak{I}$. Now we can inductively apply $rhg$ on the left of the inequality to conclude that $(rhg)^n(J)\geq (hg)^n(J)$. Thus, whenever   $(hg)^nT^{-k}(I)\geq I$ we also have that $(rhg)^nT^{-k}(I)\geq I$ by taking $J=T^{-k}(I)$. Hence, $\fd(rhg) \geq \fd(hg)$.}
\end{proof}


To complete the proof of Proposition~\ref{prop:charFDTCforSurfaces}, and hence that of Theorem \ref{thm:main}, we need the following lemma.
\begin{lemma}\label{lem:multclosed} For all $I, J$ in $\mathfrak{I}$, we have
\begin{enumerate}[(i)]
\item\label{item:P_Itotal} $P_I\cup P_I^{-1}=\G(\Sigma)$.
\item\label{item:P_Iclosed} $P_I$ is multiplicatively closed.
\item\label{item:arch} For every $g\in \G(\Sigma)$, there exists a $k_g\in\Z$ such that $T^{l}g\in P_I$ for $l\geq k_g$  and $T^{l}g\notin P_I$ for $l\leq k_{g}-1$. 
\item\label{item:P_I=P_J} $T(P_J)\subseteq P_I$.
\end{enumerate}
\end{lemma}
Properties~\eqref{item:P_Itotal},~\eqref{item:P_Iclosed}, and~\eqref{item:arch} will be used in the proof of Proposition~\ref{prop:charFDTCforSurfaces} to check that we can apply Proposition~\ref{prop:uniqunessviapositivityongroups}. Property~\eqref{item:P_I=P_J} will be used to see that there is no dependence on choices of arcs $I$ or sets of arcs $\Gamma$.
For the proof of Lemma~\ref{lem:multclosed}, we found that a beneficial perspective is to consider a certain annular cover that is topologically simple (it is homotopy equivalent to a circle), but in which it is easy to describe a lift of $T$ explicitly (something which we find not to be the case in the universal cover). We defer this perspective and the proof of Lemma~\ref{lem:multclosed} to Section~\ref{sec:annular}. With Lemma~\ref{lem:multclosed}, the proof of Proposition~\ref{prop:charFDTCforSurfaces} becomes an almost formal consequence of Proposition~\ref{prop:uniqunessviapositivityongroups}.

\begin{proof}[Proof of Proposition~\ref{prop:charFDTCforSurfaces}]
Fix an $I\in \mathfrak{I}$. Declaring $1\preceq g$ if $g\in P_I$, there is a unique way to extend this to a left-invariant relation $\preceq$. 
By Lemma~\ref{lem:multclosed},
$\preceq$ satisfies the assumptions of Proposition~\ref{prop:uniqunessviapositivityongroups}.
Indeed, by~\eqref{item:P_Itotal} $\preceq$ satisfies~\eqref{(T)} with $K=0$, by~\eqref{item:P_Iclosed} $\preceq$ is transitive,
and, for each $g\in \G(\Sigma)$ by~\eqref{item:arch} we have $g\preceq T^{k_{g^{-1}}}$ and   
$g\not\preceq T^{l}$ for $l\leq k_{g^{-1}}-1$.
Hence, there exists a unique homogeneous quasimorphism $\fd_I\colon \G(\Sigma)\to\R$ such that $\fd_I(T)=1$ and $\fd_I(P_I)\subseteq [0,\infty)$.

Let $J$ be any other element in $\mathfrak{I}$. There is a corresponding homogeneous quasimorphism $\fd_J$. However, since
$\fd_I(P_J)\subseteq \fd_I(T^{-1}(P_I))\subseteq [-1,\infty)$ by~\eqref{item:P_I=P_J}, by the first part of Proposition~\ref{prop:uniqunessviapositivityongroups} we have $\fd_J=\fd_I$.
We set $\fd\coloneqq \fd_I$ and note that $\fd(P_I)\subseteq [0,\infty)$ for all $I\in \mathfrak{I}$. 
In particular, we have $\fd(P_\Gamma)\subseteq [0,\infty)$ (since $P_\Gamma\subseteq P_I$ for all $I\in\Gamma$) and $\fd(T)=1$. Hence, it is the unique such homogeneous quasimorphism by the first part of Proposition~\ref{prop:uniqunessviapositivityongroups}.
\end{proof}

\section{Lift of a Dehn twist to the annular cover}\label{sec:annular}
In this section we prove Lemma~\ref{lem:multclosed}. The reader might find Lemma~\ref{lem:multclosed} rather intuitive. We find it important to give a careful proof since it is essentially the only topological input that makes the algebraic setup from Section~\ref{sec:qmviafloors} applicable to the setup from Section~\ref{sec:MCG}. The main input is an explicit description of $T$ on the annular cover, which we describe next.

Denote by $\Sigma^\Z$ the $\pi_1(\partial)$-cover of $\Sigma
$.
$\Sigma^\Z$ has one compact boundary component and we choose the basepoint $\overline{x_0}$ to be the unique lift of $x_0$ on said compact boundary component (except when $\Sigma$ is diffeomorphic to $S^1\times [0,1]$ in which case $\Sigma^\Z=\Sigma$ and $\Sigma^\Z$ inherits the choice of basepoint). Removing the boundary components that do not contain the basepoint from $\Sigma^\Z$ yields a surface diffeomorphic to $S^1\times[0,1)$.

In several of the following lemmas we will need a description of the lift $\widetilde{T}$ of a (positive) Dehn twist $T$ about the preferred boundary component to the cover $\Sigma^\Z$. We provide this description in Figure~\ref{figure:liftofdehntwist}. Choosing a representative of $T$ that is the identity outside a neighborhood of $\partial$ its lift in this cover acts by the identity, except in a small annular neighborhood $A$ of the compact lift of the preferred boundary component $\partial$ and the neighborhoods (diffeomorphic to $\mathbb{R}\times [0,1]$) of non-compact lifts of the preferred boundary component $\partial$. Restricted to $A,$ the representative of $\widetilde{T}$ is a (positive) Dehn twist. Inside the neighborhoods of non-compact lifts of $\partial$, the representative of $\widetilde{T}$ acts by a (positive)  translation, as shown in Figure~\ref{figure:liftofdehntwist}.

\begin{figure}[h]
\begin{tikzpicture}
\centering
\node[anchor=south west,inner sep=0](image)  at (0,0){\includegraphics[scale=1]{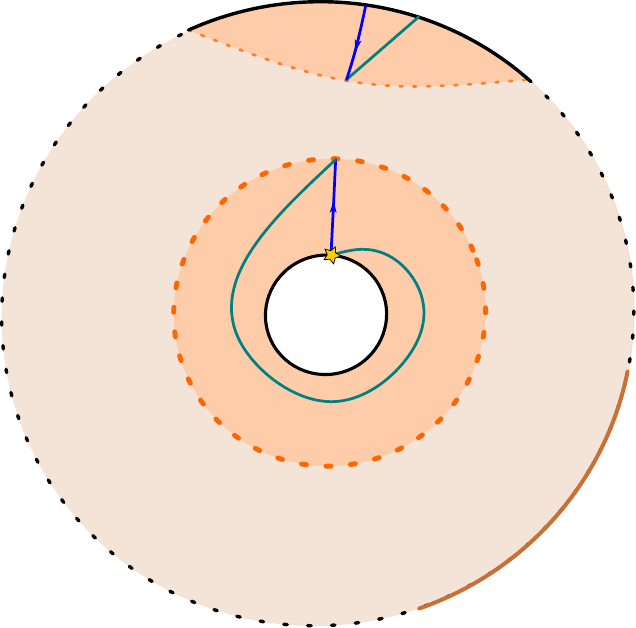}};

\begin{scope}[x={(image.south east)},y={(image.north west)}]
 \node at (0.39,0.5) {\textcolor{orange}{\Large{$A$}}};
 \node at (.43,.95){\textcolor{orange}{\Large{$L$}}};   
 \node at (.5,.45){$\widetilde{\partial_0}$} ;   
 \node at (.5, 1.03){$\widetilde{\partial_1}$};
 \node at (.95,.2){\textcolor{brown}{$\widetilde{\partial_2}$}};
 \node at (.55,.65){\textcolor{blue}{$\gamma_0$}};
 \node at (.5,.33){\textcolor{teal}{$\widetilde{T}(\gamma_0)$}};
 \node at (.53,.96) {\textcolor{blue}{$\gamma_1$}};
 
 \node at (.65,.9){\textcolor{teal}{$\widetilde{T}(\gamma_1)$}};
        
    \end{scope} 
\end{tikzpicture}
\caption{The cover $\Sigma^\mathbb{Z},$ with two non-compact lifts of boundary components drawn. The black circle $\widetilde{\partial_0}$ and black line $\widetilde{\partial_1}$ are lifts of the preferred boundary $\partial$. The brown line $\widetilde{\partial_2}$ is a lift of a different boundary component. The region $A$ is a small neighborhood of $\widetilde{\partial_0}$, and $L$ is a small neighborhood of $\widetilde{\partial_1}$. The lift $\widetilde{T}$ of $T$ in this cover acts by the identity, away from of a small annular neighborhood of the compact lift of the boundary and the $\mathbb{R}\times I$ neighborhood of non-compact lifts of the preferred boundary component.}
\label{figure:liftofdehntwist}
\end{figure}
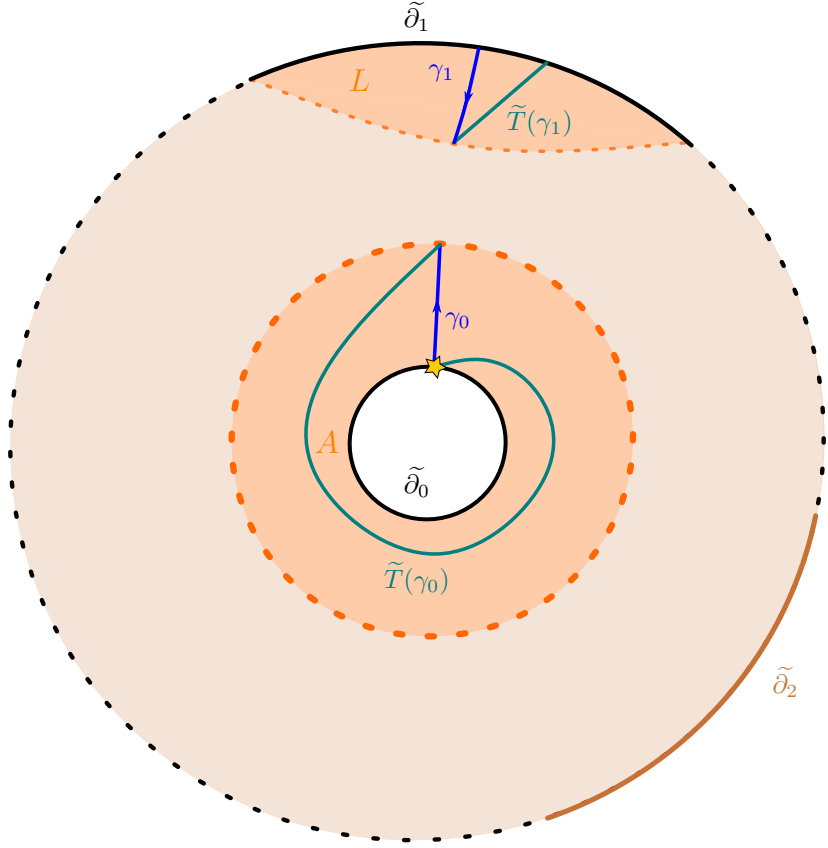

 Let $\widetilde{I}$ and $\widetilde{J}$ denote the lifts of $I$ and $J$ to $\Sigma^\Z$ with start point $\overline{x_0}$.
(In the definition of homotopy classes of arcs on $\Sigma^\Z$, e.g.~$\widetilde{I}$ and $\widetilde{J}$, homotopies are taken to fix those endpoints that are on the preimage of those boundaries that remain fixed in the isotopies in the definition of $\G(\Sigma)$.) For representatives of $\widetilde{I}$ and $\widetilde{J}$ that intersect transversally and minimally, the signed count of intersections in the interior of $\Sigma^\mathbb{Z}$ is equal to the total number of intersections in the interior of $\Sigma^\mathbb{Z}$ up to sign, since the interior of $\Sigma^\mathbb{Z}$ is an annulus. We denote this signed count of intersections (in the interior) of 
representatives of $\widetilde{I}$ and $\widetilde{J}$ that intersect transversally and minimally by 
 $i\left(\widetilde{I},\widetilde{J}\right)\in \mathbb{Z}$.

Finally, we note that if two distinct arcs with the same starting point intersect minimally in the annular cover, their lifts to the universal cover are disjoint away from the starting point. Indeed, if they were not disjoint, one would find a bigon in the annular cover using an innermost disk argument (see e.g.~\cite[Lemma 1.8]{FarbMargalit_12_APrimerOnMCG}), in contradiction to the minimality of the intersection. Hence, being to the left or right (in the sense of Definition \ref{def:totheright}), can be determined in the annular cover by comparing the tangent vectors of the two arcs near their starting point. We now have the necessary setup to prove a series of topological lemmas, with goal of proving Lemma \ref{lem:multclosed}.

\begin{lemma}\label{lemma:intersectionsign} If $i\left(\widetilde{J},\widetilde{I}\right)<0$ then $J<I$, and if $i\left(\widetilde{J},\widetilde{I}\right)>0$ then $J>I.$
\end{lemma}

\begin{proof}
Assume $i\left(\widetilde{J},\widetilde{I}\right)<0$. With the intersections of $\widetilde{I}$ and $\widetilde{J}$ as in Figure~\ref{fig:Annulusorientations},
\begin{figure}[h]
\centering
\captionsetup[subfigure]{labelformat=empty}

\begin{tikzpicture}[scale=1.0]
	\node[anchor=south west,inner sep=0] (image) at (-2.7,0) {\includegraphics[width=.4\textwidth]{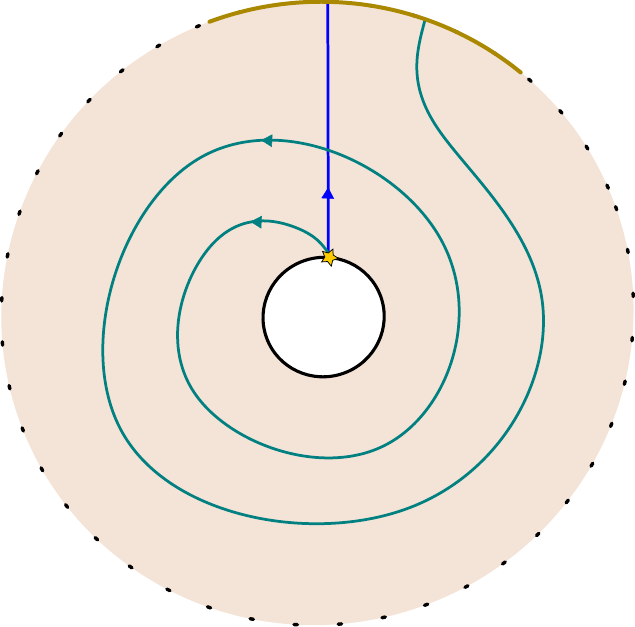} };
	\begin{scope}[x={(image.south east)},y={(image.north west)}]
	  \node at (1, .85){ $\textcolor{blue}{\widetilde{I}}$};
        \node at (0.3, .8){ $\textcolor{teal}{\widetilde{J}}$};
           
	\node[inner sep=0] (image) at (3.2,.5) {\includegraphics[width=.4\textwidth]{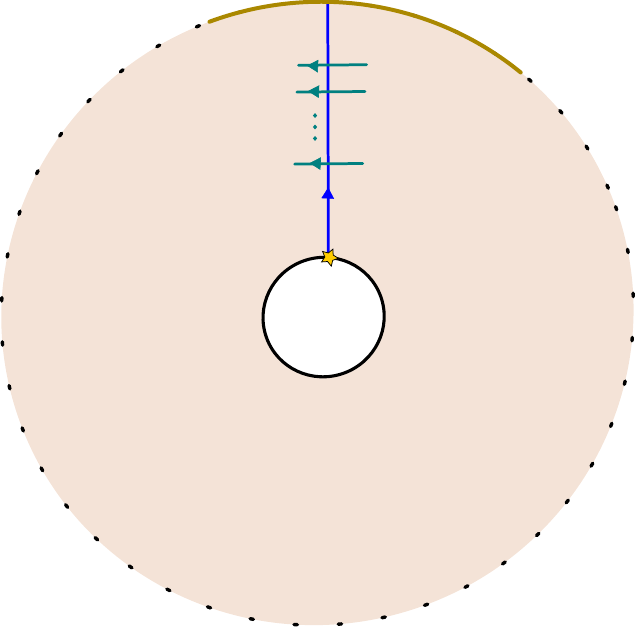} };
	  \node at (3.6, .65){ $\textcolor{blue}{\widetilde{I}}$};
	  \node at (3.9, .85){ $\textcolor{teal}{\widetilde{J}}$};
	\end{scope}
\end{tikzpicture}

\caption{On the left we give an example of $\widetilde{I}$ and $\widetilde{J}$ with $i\left(\widetilde{J},\widetilde{I}\right)=-1$ and $J<I$. On the right we illustrate the general case $i(\widetilde{J},\widetilde{I})<0$ and note that extending the teal arcs to a properly embedded arc $\widetilde{J}$, which has minimal intersection with $\widetilde{I}$, forces $J<I$.}
\label{fig:Annulusorientations}
\end{figure}
the only possibilities to complete $\widetilde{J}$ with minimal intersection with $\widetilde{I}$ forces $\widetilde{J}$ to be strictly to the left of $\widetilde{I}$. This is equivalent to $J<I$.  The case $i\left(\widetilde{J},\widetilde{I}\right)>0$ is obtained by reflecting Figure~\ref{fig:Annulusorientations}. 
\end{proof}
Next we note that $T$ is strictly $\partial$-right-veering.
\begin{lemma}\label{lem:fulltwistcomparison}For all $I,J\in\mathfrak{I}$, $J\geq I$ implies $T(J)>I$.
\end{lemma}

\begin{proof}
It is not hard to check $T(J)>J$: it suffices to see that $\widetilde{T}(\widetilde{J})$ is strictly to the right of $\widetilde{J}$, which is readily seen from our explicit description of $\widetilde{T}(\widetilde{J})$ applied to an embedded arc representing $\widetilde{J}$. 
By transitivity we have that $T(J)>J\geq I$.
\end{proof}
\begin{lemma}\label{lem:manytwistincreasei} For all $I$ and $J$ there exists an $n$ such that $i\left(\widetilde{T^n(J)},\widetilde{I}\right)>0$ and $i\left(\widetilde{T^{-n}(J)},\widetilde{I}\right)<0$.
\end{lemma}
\begin{proof}
For the sake of visualization, note that a representative of $\widetilde{I}$ can be taken to be a radius of the interior of $\Sigma^{\mathbb{Z}}$. Suppose that $\widetilde{I}$ ends on a lift of some boundary component, say $\widetilde{\partial_I}$, and let $L_I$ be a neighborhood of this lift. We consider the intersection number between $\widetilde{I}$ and $\widetilde{T^n (J)}=\widetilde{T}^n(\widetilde{J})$ in three regions of $\Sigma^{\mathbb{Z}}$, namely $A,$ $L_I$ and the complementary region. Comparing the intersections between (representatives of) $\widetilde{I}$ and $\widetilde{J}$ and the ones between $\widetilde{I}$ and $\widetilde{T}^n(\widetilde{J})$, we see that they are the same outside $L_I\cup A$, the number changes by at most $1$ in $L_I$, and in $A$ the algebraic intersection number can be made arbitrarily large by increasing $n$.

For the second statement, observe that we just showed there exists an $n$ such that  $i\left(\widetilde{T^n(I)},\widetilde{J}\right)>0$ by exchanging the roles of $I$ and $J$. Hence, \[i\left(\widetilde{T^{-n}(J)},\widetilde{I}\right)=
i\left(\widetilde{T}^{-n}(\widetilde{J}),\widetilde{I}\right)=
i\left(\widetilde{J},\widetilde{T}^n(\widetilde{I})\right)=i\left(\widetilde{J},\widetilde{T^n(I)}\right)=-i\left(\widetilde{T^{n}(I)},\widetilde{J}\right)<0.\qedhere\]
\end{proof}
\begin{lemma}\label{lem:IcomparestoJ} For all $I, J$ in $\mathfrak{I}$, there exists a $k\in\Z$ such that
$T^{l}(J)<I$ for $l<k$ and $I\leq T^l(J)$ for $l\geq k$. If $I=J$, then $k=0$.
\end{lemma}

\begin{proof}
Let $S\coloneqq \left\{l\in\Z\mid I\leq T^l(J)\right\}$. By Lemma~\ref{lem:fulltwistcomparison}, if $l\in S$ then all $l'>l$ are in $S$. $S$ is non-empty since there is an $n$ such that
$i\left(\widetilde{T^n(J)},\widetilde{I}\right)>0$
by Lemma~\ref{lem:manytwistincreasei}; hence, $T^n(J)>I$ by Lemma~\ref{lemma:intersectionsign}. Finally, $S$ is bounded below since there is an $n$   
such that
$i\left(\widetilde{T^{-n}(J)},\widetilde{I}\right)<0$
by Lemma~\ref{lem:manytwistincreasei}; hence, $T^{-n}(J)<I$ by Lemma~\ref{lemma:intersectionsign}, thus $-n\notin S$. Moreover, if $N\leq -n$ then $N\notin S$ since $T^N(J)=T^{N+n}T^{-n}(J)\leq T^{-n}(J)<I$. 
Therefore, we can take $k\coloneqq \min S$.

In case $I=J$, clearly we have $0\in S$, but $T^{-n}(I)<I$ (hence $-n\notin S$) for all $n>0$.
\end{proof}

\begin{proof}[Proof of Lemma~\ref{lem:multclosed}]
\eqref{item:P_Itotal} follows since $g(I)\leq I$ or $I\leq g(I)$ for each element $g\in \G(\Sigma)$ ($\leq$ is a total order by Lemma~\ref{lem:<=istotal}).

\eqref{item:P_Iclosed} follows from $\leq$ being left-invariant and transitive. Indeed, if $g,h\in P_I$, which is equivalent to $I\leq g(I)$ and $I\leq h(I)$) but also equivalent to $I\leq g(I)$ and $g(I)\leq gh(I)$, then $I\leq gh(I)$, which is equivalent to $gh\in P_I$.

\eqref{item:arch} We set $J\coloneqq g(I)$ and let $k_g\coloneqq k$, where $k$ is the integer guaranteed to exist by Lemma~\ref{lem:IcomparestoJ}. Hence, $T^{l} (J)\geq I$ (meaning $T^lg\in P_I$) for $l\geq k_g$, while 
$T^{l} (J)<I$ (meaning $T^{l}g\notin P_I$) for $l\leq k_g-1$.

\eqref{item:P_I=P_J}
Pick $g\in P_J$, meaning $J\leq g(J)$. Let $k$ be the integer such that $T^{k-1}(I)<J\leq T^k(I)$, which exists by Lemma~\ref{lem:IcomparestoJ}.
Hence, $T^{k-1}(I)<J$ and $g(J)\leq T^kg(I)$, which yields (by $J\leq g(J)$) that
$T^{k-1}(I)<TgT^{k-1}(I)$. The latter in particular means $T(g)\in P_{T^{k-1}(I)}$. Noting that $P_{T^lI}=P_I$ for all $l\in \Z$ since $I\leq g (I)$ if and only if $T^l(I)\leq T^lg(I)=gT^l(I)$, we have shown that $Tg\in P_{I}$.
\end{proof}

\section{Surfaces of infinite type for which all reals are realized}\label{sec:irrationality}
In this section, we provide surfaces for which the DTC is surjective. In fact, we provide two constructions of such surfaces. Two base examples, one for each construction, to keep in mind are
\[\mathbb{D}\setminus C\quad\text{and}\quad \R^2\setminus \left(\mathbb{D}^\circ\cup \{(0,n)\mid 2\leq n\in \Z\}\right),\] where $\mathbb{D}$ and $\mathbb{D}^\circ$ denote the closed and open unit disks in $\R^2$, respectively, and $C$ is a Cantor set in $\mathbb{D}^\circ\subset \mathbb{D}$. In terms of the classification of surfaces, $\mathbb{D}\setminus C$ is
the genus zero surface with boundary a circle and space of ends homeomorphic to a Cantor set, and
$\R^2\setminus \left(\mathbb{D}^\circ\cup \{(0,n)\mid 2\leq n\in \Z\}\right)$ is
the genus zero surface with boundary a circle and space of ends homeomorphic to $\{1/n\mid n\in \N\}\cup \{0\}\subseteq \R$.

\begin{restatable}{thm}{allofR}
\label{thm:allofRprecise}
Let $\Sigma$ be a surface with a fixed compact boundary component $\partial$. 
\begin{enumerate}[(a)]
\item \label{it:b}
We fix a 
surface $S$ that is not the two sphere.
If $\Sigma$ is obtained from $\R^2\setminus \mathbb{D}^\circ$ by connect summing on a copy of $S$ at each integer lattice point in the interior of $\R^2\setminus \mathbb{D}^\circ$, and $\partial$ is the boundary component stemming from $\R^2\setminus \mathbb{D}^\circ$, then $\fd\colon\G(\Sigma,\partial)\to\R$ is surjective.

\item\label{it:a}
If $\Sigma$ is obtained from taking the connect sum of $\mathbb{D}\setminus C$---where $C$ denotes a Cantor set in the interior of $\mathbb{D}$---and any other 
surface, and $\partial$ is the boundary component stemming from $\mathbb{D}\setminus C$, then $\fd\colon\G(\Sigma)\to\R$ is surjective.
\end{enumerate}
\end{restatable}
In particular, this establishes Theorem~\ref{thm:allofR}.
\begin{proof}[Proof of Theorem~\ref{thm:allofR}] Clearly, Theorem~\ref{thm:allofR} follows from Theorem~\ref{thm:allofRprecise}.
\end{proof}

\begin{Remark}  We remark that all $\Sigma$ from family~\eqref{it:b} are, after capping off all boundary components with discs, \emph{self-similar}; see \cite{vlamis2023homeomorphism}. Further, while not all
self-similar surfaces arise, all \emph{perfectly self-similar surfaces} arise by capping off all boundary components of a $\Sigma$ from family~\eqref{it:b} with discs; see~\cite[Proof of Theorem 9.2]{vlamis2023homeomorphism}.

\end{Remark}


Towards establishing Theorem~\ref{thm:allofRprecise}, we start with the construction that, as far as we are aware, produces a class of elements of big mapping class groups that were not studied previously; see Section~\ref{ssec:plane-discs}. In Section~\ref{ssec:disc-cantorset}, we use a construction of mapping classes that we believe to be well-known in the big mapping class group community. Finally, in Section~\ref{subsec:thmallofR}, we use the two base examples from Section~\ref{ssec:plane-discs} and Section~\ref{ssec:disc-cantorset} to establish Theorem~\ref{thm:allofRprecise}. 

\subsection{The wagon wheel map} \label{ssec:plane-discs}
To simplify notation in this section, we identify $\mathbb{R}^2$ with $\mathbb{C}$. For $v\in\mathbb{C},$ we write
\[B_v(\epsilon)\coloneqq \left\{w\in \mathbb{C} \mid |v-w|\leq\epsilon\right\}.\]

\begin{prop}\label{prop:allofRPlane-discs}
Let $\Sigma_P\coloneqq \C\setminus \bigcup_{\{v\in\C\mid \Re(v),\, \emph{\text{Im}}(v)\in\Z\}} B_v^\circ(\frac{1}{3})$ and let $\partial$ be one of the boundary components of $\Sigma_P$. Then the $\DTC$ $ \fd\colon \G(\Sigma_P,\partial) \to \R$ is surjective. 
\end{prop} 

For the proof we explicitly construct, for each $\lambda\in\R\setminus\Q$, a self-diffeomorphism $\phi$ of $\Sigma_P$ (and the corresponding mapping class $\Phi$) with $\fd(\Phi)=\lambda$, which we call \emph{wagon wheel map of rotation $\lambda$}.
\begin{proof}[Proof of Proposition~\ref{prop:allofRPlane-discs}]

We will show in Section \ref{section:rationals} that $\Q\subseteq \fd(\G(\Sigma_P,\partial))$. Hence, we fix $\lambda\in \R\setminus \Q$ and show that there exists a $g\in \G(\Sigma_P,\partial)$ with $\fd(g)=\lambda$.

For this we consider the following diffeomorphic copy of $\Sigma_P$, which we denote by $\Sigma$. Set $v_k\coloneqq (2.5+|k|)e^{2\pi i k\lambda}$ for $k\in\Z$ and choose $\epsilon_k>0$ sufficiently small; concretely, $\epsilon_k<\lambda-\lfloor\lambda\rfloor, \lceil\lambda\rceil-\lambda$ for all $k\in\Z$ such that $2\epsilon_k + 2\epsilon_{-k}<|v_k-v_{-k}|$ for $k\in\Z\setminus\{0\}$ will suffice as we will see below. We let $D_k$ denote $B_{v_k}(\epsilon_{k})$ and set
\[\Sigma\coloneqq \C\setminus \left(B_{0}^\circ(1)\cup\bigcup_{k\in\Z} D^\circ_k\right)\]
and $\partial\coloneqq \partial B_0(1)$; see Figure~\ref{fig:SigmaP} for an illustration of $\Sigma$.
\begin{figure}[h]
\centering
\includegraphics[width=.8\textwidth]{
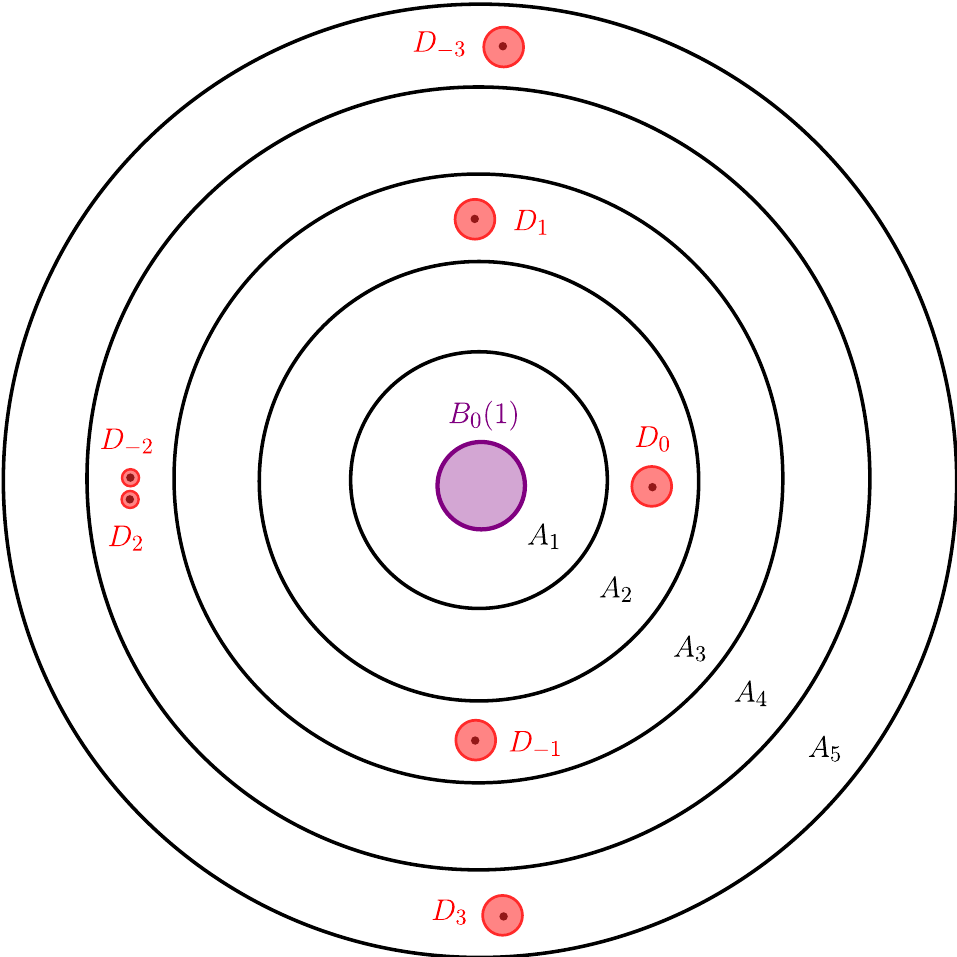} 
\caption{The surface $\Sigma$  when $\lambda$ is just larger than $1/4$ is obtained by removing the shaded disk interiors. The preferred boundary $\partial$ is $\partial B_0(1).$}
\label{fig:SigmaP}
\end{figure}

{
In what remains of this proof we orient $\mathbb{C}$ {opposite} to the standard convention for a technical reason.
}

We define a diffeomorphism $f\colon \C\setminus B_{0}^\circ(1)\to \C\setminus B_{0}^\circ(1)$ by defining it on each annulus $A_n\coloneqq \{z\in \C\mid n\leq |z|\leq n+1\}$ for $n\in \N$. The map $f$ will restrict to a diffeomorphism on $\Sigma$.  We now describe the construction of $f$.

On \textbf{$A_1$} we set $f\left((t+1)s\right)\coloneqq (t+1)se^{2\pi i t\lambda}$ for $t\in[0,1]$ and $s\in S^1\subset\C$.
On \textbf{$\bigcup_{n\geq 2} A_n$} we define $f|_{\bigcup_{n\geq 2} A_n}\colon \bigcup_{n\geq 2} A_n\to \bigcup_{n\geq 2} A_n$ as the composition of two maps $R_\lambda$ and $P$ (in fact, $f(A_2)\subseteq A_2\cup A_3$ and $f(A_n)\subseteq A_{n-1}\cup A_n\cup A_{n+1}$ for $n\geq 3$), which we describe below.

Firstly, we let $R_\lambda\colon \C\to\C$ be the rotation by $2\pi\lambda$, i.e.~$R_\lambda(z)=ze^{2\pi i \lambda}$. 
Secondly, $P$ is defined to be the identity outside compact regions $K_k$ defined as follows. We set $K_k\subseteq\C$ to be a small closed neighborhood that contains the convex hull of $R_\lambda(D_k) \cup D_{k+1}$. We choose the $K_k$ so that they do not overlap; this is possible by our choice of $\epsilon_k$. On $K_k$ we define $P$ to be a diffeomorphism that maps $R_\lambda(D_k)$ to $D_{k+1}$
given by a time 1 map of an appropriate vector field that is transverse to the circles around the origin $0\in\C$ and zero on the boundary of each $K_k$. Notice that this vector field is outwards transverse to the circles around the origin in $K_k$ for $k\geq 0$ and inwards transverse to the circles around the origin in $K_k$ for $k< 0$. In intuitive terms, $P$ pushes the disk $R_\lambda(D_k)$ onto the disk $D_{k+1}$ within $K_k$ in the most straightforward fashion possible. 

Outside of $A_1$, we set $f$ to be the composition of $P$ and $R_\lambda$, that is, $f(z)\coloneqq P(R_\lambda(z))$; see Figure~\ref{fig:Wagonwheel}.
\begin{figure}[h]
\centering
\includegraphics[width=1\textwidth]{
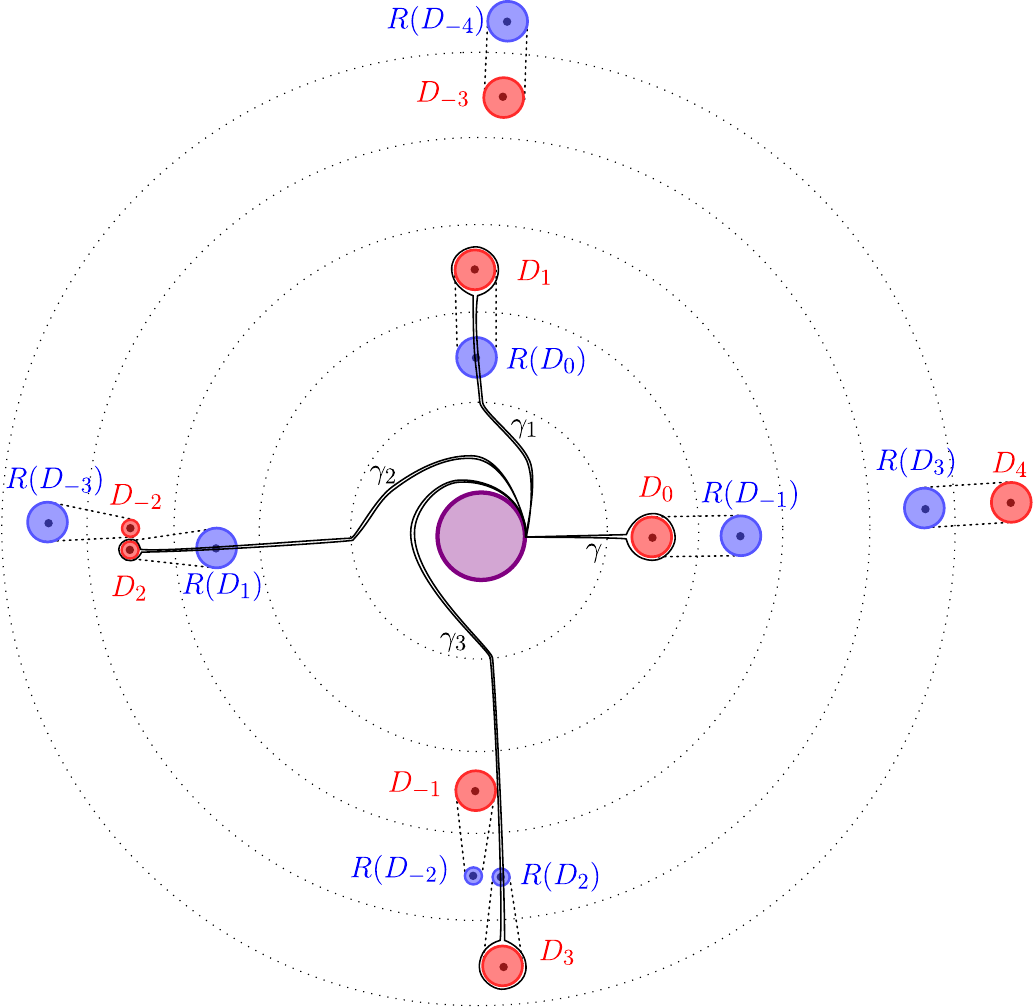} 
\caption{We illustrate the action of the wagon wheel map $\Phi$ of rotation $\lambda$ on $\Sigma$  when $\lambda$ is just larger than $1/4$. The map $\Phi$ first rotates by angle $\lambda$ in each annulus $A_i$ for $i\geq 2$. The images of the $D_i$ under this rotation are shown in blue. Then $\Phi$ `pushes' the blue disks onto the red disks along the regions $K_k$ indicated by the dotted lines. Also drawn is an embedded arc homotopic to $\gamma$ as well as its first three iterates under $\Phi$.}
\label{fig:Wagonwheel}
\end{figure}
This concludes the construction of $f$.

We define $\phi$ on $\Sigma$ to be the restriction of $f$ to $\Sigma$. This restriction is well-defined because $f$ fixes $\bigcup_k D_k^\circ$ as a set. We claim that the mapping class $\Phi \coloneqq [\phi]$ satisfies $\fd(\Phi)=\lambda$.

To see this claim, choose $I=[\gamma]$, where
$\gamma$ is the arc that starts at 1, follows the straight line to 2.25, then travels counter-clockwise once around the boundary of $D_0$ and subsequently follows the straight line back to 1; see Figure~\ref{fig:Wagonwheel}. To be explicit, $\gamma$ may be parametrized as follows:
\[\gamma\colon [0,1]\to \Sigma, t\mapsto 
\left\{\begin{array}{ll}
			1+5t & \text{if $t\in[0,0.25]$}\\
        	2.5+0.25e^{2\pi i 2t} & \text{if $t\in[0.25,0.75]$}\\
        	2.25-5(t-0.75) & \text{if $t\in[0.75,1]$}
		\end{array}\right .\]
		
Next we describe a path $\gamma_n\colon [0,1]\to\Sigma$ that describes the image of $I$ under $\Phi^n$, i.e.~we have $\Phi^nI=[\gamma_n]$. 
The path $\gamma_n$ is essentially given by
		the path $\gamma'_n\colon [0,1]\to \C$ given as the consecutive composition of the following five paths:
the arc that goes from $1$ to $2e^{2\pi i n\lambda}$ (wrapping around the origin by $n\lambda$) as parametrized by $(1+t)e^{2\pi i n\lambda t}$ for $t\in[0,1]$, the straight arc on the ray of angle $2\pi n\lambda$ from $2e^{2\pi i n\lambda}$ to $\partial D_n$, a loop that goes once around $\partial D_n$ counter-clockwise, the reverse of the straight arc described above, and finally the reverse of the arc described above that goes from $2e^{2\pi i n\lambda}$ to $1$ (wrapping around the origin by $-n\lambda$); see Figure~\ref{fig:Wagonwheel}.
The careful reader notices that the image of $\gamma_n'$ does not land in $\Sigma$. We change $\gamma'_n$ as follows to obtain $\gamma_n$: we choose $\gamma_n(t)$ as $\gamma'_n(t)$, except when $\gamma'_n(t)\in D_k^\circ$ for some $k\in\Z$ in which case we replace the straight subarc of $\gamma'_n$ that hits $D_k^\circ$ with the shortest circular arc (since $\gamma'_n$ never hits any $v_k$ there is a well-defined shortest one) that goes around the $\partial D_k$.
From our definition of $f$, we see that $\Phi^nI=[\gamma_n]$.
		
We show that
\begin{equation}\label{eq:plane-disc}
T^{\lfloor n\lambda\rfloor} \prec_I \Phi^n\prec_I T^{\lceil n\lambda\rceil}\end{equation}
for all $n\in\N$. This suffices since $a\prec_I b$ implies $\fd(a)\leq \fd(b)$ and thus it will follow from ~\eqref{eq:plane-disc} that
 \[\lfloor n\lambda\rfloor=\fd(T^{\lfloor n\lambda\rfloor})\leq \fd(\Phi^n)\leq\fd(T^{\lceil n\lambda\rceil})=\lceil n\lambda\rceil;\] hence, using homogeneity, $\frac{\lfloor n\lambda\rfloor}{n}\leq 
\fd(\Phi)\leq\frac{\lceil n\lambda\rceil}{n}.$

To see~\eqref{eq:plane-disc}, we note that it is equivalent to show $T^{-\lceil n\lambda\rceil}\Phi^n(I)<I<T^{-\lfloor n\lambda\rfloor}\Phi^n(I)$.
We write $J\coloneqq T^{-\lfloor n\lambda\rfloor}\Phi^n(I)$ and show $I<J$. ($T^{-\lceil n\lambda\rceil}\Phi^n(I)<I$ follows by a symmetric argument.)

We realize $T$ as the homotopy class of $D_\partial\colon \Sigma\to \Sigma$ defined by
\[D_\partial(z)=
\left\{\begin{array}{ll} ze^{2\pi i |z|}&\text{ if $1\leq |z|\leq 2$}\\
z &\text{ if $|z|\geq 2$}\end{array}\right. .\]

{Note, if we had oriented $\Sigma$ as a subset of $\mathbb{C}$ with the standard orientation, $T$ would be a left-veering map and have $\omega(T)=-1$. However, since our conventions flip the orientation, $T$ is right-veering.}

Hence, $J=[\beta_n]$, where $\beta_n=D_\partial^{-\lfloor n\lambda\rfloor}\circ\gamma_n$. To be explicit, $\beta_n$ equals $\gamma_n$ outside of $A_1$, while in $A_1$, $\beta_n$ starts from $1$ and goes to $2e^{2\pi i n\lambda}=2e^{2\pi i (n\lambda-\lfloor n\lambda\rfloor)}$ by wrapping $n\lambda-\lfloor n\lambda\rfloor$ around the origin, then leaves $A_1$ and when it reenters $A_1$ it goes from $2e^{2\pi i (n\lambda-\lfloor n\lambda\rfloor)}$ to $1$ by wrapping $-(n\lambda-\lfloor n\lambda\rfloor)$ around the origin. Concretely, the first part (out of five) of $\beta_n$ is parametrized by $(1+t)e^{2\pi i (n\lambda-\lfloor n\lambda\rfloor) t}$ for $t\in[0,1]$, while the last part is parametrized by the reverse of this. We see that the images of $\beta_n$ and $\gamma$ are disjoint except at $1$ (which is the start-point and end-point for both). Hence, since $\beta_n$ and $\gamma$ are not homotopic relative endpoints, we have that their lifts $\widetilde{\beta_n}$ and $\widetilde{\gamma}$ to the universal covering $\pi\colon\widetilde{\Sigma}\to \Sigma$ (with respect to a some base point $\widetilde{x_0}$ with $\pi(\widetilde{x_0})=1$) are disjoint away from $\widetilde{x_0}$. By Definition~\ref{def:totheright}, using the lifts $\widetilde{\beta_n}$, $\widetilde{\gamma}$ and  $n\lambda-\lfloor n\lambda\rfloor>0$, we find $J>I$.
\end{proof}

\subsection{Irrational rotations of the disk minus a Cantor set}\label{ssec:disc-cantorset}  

Let $\D$ be the closed disk unit disk in $\C$ and let $C$ be a Cantor set in the interior of $D$. We set $\Sigma_C\coloneqq \D\setminus C$. For this surface we will show that $\omega: \G(\Sigma_C) \to\mathbb{R}$ is surjective.

To do this, we construct for each $r\in\R/\Z$ an element $g\in \G(\Sigma_C)$ with $\fd(g)\in r$ using a homeomorphism of the circle with rotation number $r$ that has an invariant Cantor set. For $r\in \Q/\Z$, for the latter, one may choose a rotation. For irrational $r\in \R/\Z$, while envisioned by Poincar\'e, an explicit construction of such a map is due to Denjoy; see~\cite[Paragraph after Theorem~5.9]{Ghys_GroupsActingOnCircle} for an intuitive description and~\cite[Proposition~12.2.1]{KH_Intro_to_dynamical} for a more rigorous definition. From such a homeomorphism of the circle one finds an interesting mapping class of $\Sigma_C$ by viewing the circle as a subset of the interior of $\D$ and extending the above described homeomorphism with rotation number $r$ to all of $\Sigma_C$, which yields
a well defined homeomorphism when choosing $C$ to coincide with the invariant Cantor set in $S^1$.
Coming from the perspective of the DTC, it was rather natural to consider these mapping classes as potential examples with irrational DTC. However, we make no claim of originality here: this idea of producing mapping classes from Denjoy's example have been known to experts on big mapping class groups; see e.g.~\cite[Proof of Theorem~5.1]{CalegariChen22}. They are called \emph{irrational rotations of rotation number $r$}; see~\cite[Section 4]{BFT}.


\begin{prop}\label{prop:allofRCTree}  Let $\Sigma_C$ be closed disk from which we remove a Cantor set in the interior. Then $\fd(\G)=\R$.
\end{prop}


\begin{proof}[Proof of Proposition~\ref{prop:allofRCTree}]
Fix $r\in\R/\Z$. We show that $\G(\Sigma_C)$ for this surface contains an element $g$ with $\fd(g)\in r$. Since $\fd(T^kg)=\fd(g)+k$ for all $k\in\Z$, this suffices.

Let $\phi_r\colon S^1\to S^1$ be an orientation preserving homeomorphism with rotation number $r$ and an invariant Cantor set $C\subset S^1$. {Here, we use the classical definition of rotation number for a homeomorphism of the circle, namely $r(\phi_r)=\lim_{n\to\infty} \frac{\widetilde{\phi_r}^n(0)}{n}+\Z$, where $\widetilde{\phi_r}\colon\R\to\R$ is a lift of $\phi_r$ with respect to the universal covering $\R\to S^1$.} 

We consider the following diffeomorphic copy of $\Sigma_C$, which we denote by $\Sigma$:
the closed disk $D_2$ of radius $2$ around $0$ in $\C$ from which we remove $C\subset S^1\subset\C$, i.e.~$\Sigma=D_2\setminus C\subset D_2\subset \C$.
Let $\Phi_r\colon D_2\to D_2$ be any homeomorphism with $\Phi_{r|S^1\to S^1}=\phi$ that is the identity on the boundary of $D_2$. Note that all choices of $\Phi_r$ are isotopic up to composing with Dehn twists about the boundary. We set $g_r\coloneqq [\Phi_{r|\Sigma\to\Sigma}]$.
The careful reader notices that we defined $\G(\Sigma)$ using diffeomorphisms but $\Phi_{r|\Sigma\to\Sigma}$ is only a homeomorphism. However, we may and do assume that $\Phi_{r|\Sigma\to\Sigma}$ is a diffeomorphism,
since all homeomorphisms of surfaces are isotopic to a diffeomorphism (see~\cite[Theorem B]{Hatcher_torustrickforsurfaces} for the arguably easiest proof, but compare also with~\cite[Sect.~6: Theorem~4]{Moise_book77} for a classical albeit PL-reference).
For the reader preferring hands on constructions,
one may arrange by hand that $\Phi_{r|\Sigma\to\Sigma}$ is smooth by making use of the rotational symmetry and applying convolutions with mollifiers to concentric circles. 

We claim that $\fd(g_r)\in r$. To see this, take $I=[\gamma]$, where $\gamma\subset \Sigma$ is a diameter of $D_2$ that misses $C\subset S^1\subset D_2$. To make this more concrete, we assume without loss of generality, that $\pm 1\in \Sigma$ (in other words, we arrange for $\pm 1\not \in  C\subset S^1$) and that $\pm 1$ lie in different connected components of $S^1\setminus C$.
Take $I$ to be the homotopy class of the arc $\gamma$ going straight from $2$ to $-2$ (recall $\pm 2\in \partial \Sigma\subset D_2\subset \mathbb{C}$).
Explicitly, we set $\gamma\colon [0,1]\to \Sigma,$ parametrized by $\gamma(t)\mapsto 2-4t$. The fact that $\pm 1$ lie in different connected components of $S^1\setminus C$ assures that $\gamma$ is not boundary parallel relative endpoints.

Consider the universal cover $\pi\colon\widetilde{\Sigma}\to \Sigma$ and identify in $\widetilde{\Sigma}$  the preimage of $\{z\in \Sigma: |z|\geq1\}$ with $N\coloneqq([0,1]\times \R)\setminus \left(\{1\}\times\widetilde{S^1\setminus C}\right)$
such that $\pi((s,x))=(2-s)e^{2\pi ix}\in \Sigma$ for $(s,x)\in N\subset\widetilde{\Sigma}$. (Here, $\widetilde{S^1\setminus C}$ is the preimage of $S^1\setminus C$ under the universal covering $\R\to S^1, s\mapsto e^{2\pi i x}$.)
In particular, $\widetilde{\gamma}$---the lift of $\gamma$ to $(0,0)\in N\subset\widetilde{\Sigma}$---satisfies $\widetilde{\gamma}(t)=(4t,0)\in N$ for $t\leq \frac{1}{4}$.
We denote by $\widetilde{\Phi_r}$ the lift of $\Phi_r$ that is the identity on $\partial\widetilde{\Sigma}=\{0\}\times \R$. For all $n\in\N$, let $\lambda_n\in\R$ be such that $\widetilde{\Phi_r}^n((1,0))=(1,\lambda_n)$. By definition of rotation number of $\Phi_r$, we have $r=[\lambda]$, where $\lambda\coloneqq \lim_{n\to\infty}\frac{\lambda_n}{n}$. We claim that $\fd(g_r)=\lambda$.

To see this claim, it suffices to show $T^{\lfloor \lambda_n \rfloor -1}\prec_I g_r^n \prec_I T^{\lceil \lambda_n \rceil +1}$, which is equivalent to
$T^{-\lceil \lambda_n \rceil -1}g_r^n(I)<I< T^{-\lfloor \lambda_n \rfloor +1}g_r^n(I)$.
We argue that $I<T^{-\lfloor \lambda_n \rfloor +1}g_r^n(I)$ (the other inequality follows by a symmetric argument). We write $J\coloneqq T^{-\lfloor \lambda_n \rfloor +1}g_r^n(I)$ for brevity.

Let $f\colon \Sigma\to \Sigma$ be an explicit representative of $T$ that is the identity outside a small neighbourhood of the boundary.
We consider lifts of $f^{-\lfloor \lambda_n \rfloor +1}\circ\Phi_r^n(\gamma)$, denoted by $\widetilde{\gamma_n}$,
starting at $(0,0)$, and note that $\widetilde{\gamma_n}(\frac{1}{4})=(1,\lambda_n-\lfloor \lambda_n \rfloor+1)$. In particular, $\widetilde{\gamma}(\frac{1}{4})$ and $\widetilde{\gamma_n}(\frac{1}{4})$ are on different components of $\{1\}\times \widetilde{\R\setminus C}$ and $\widetilde{\gamma_n}(\frac{1}{4})$ is to the right of $\widetilde{\gamma}(\frac{1}{4})=(1,0)$ on $\{1\}\times \R$ since $(\lambda_n-\lfloor \lambda_n \rfloor+1)-0\geq 1$.
Given that $\widetilde{\gamma}(\frac{1}{4})$ and $\widetilde{\gamma_n}(\frac{1}{4})$ are on different components of $\{1\}\times \widetilde{\R\setminus C}$,
we have $\widetilde{\gamma}$ and $\widetilde{\gamma_n}$ are disjoint away from $N$; in other words, $\widetilde{\gamma}(t)\neq\widetilde{\gamma_n}(t)$ for $t\geq \frac{1}{4}$. Here, we use that $\gamma$ enters $D_1\setminus C=\{z\in \Sigma\mid |z|\leq 1\}$ at a different component of $S^1\setminus C$ than it exits $D_1\setminus C$.

We homotope $\widetilde{\gamma_n}$ to an arc we denote by $\widetilde{\gamma_n}'$ and define as follows. $\widetilde{\gamma_n}'$ is equal to $\widetilde{\gamma_n}$ for $t\geq \frac{1}{4}$ and the it is equal to the straight line segment in $N$ from $(0,0)$ to $\widetilde{\gamma_n}(\frac{1}{4})$ for $t\leq\frac{1}{4}$
(i.e.~$\widetilde{\gamma_n}'(t)=(4t,4t+\lambda_n-\lfloor \lambda_n \rfloor+1)\in N\subset\widetilde{\Sigma}$ for $t\leq\frac{1}{4}$). 
By construction we have that $\widetilde{\gamma}$ and $\widetilde{\gamma_n}'$ are disjoint except at the start point $\widetilde{x_0}\coloneqq (0,0)$.
Hence, $\widetilde{\gamma}$ and $\widetilde{\gamma_n}'$ are representatives of $\widetilde{I}$ and $\widetilde{J}$ that are disjoint away from $\widetilde{x_0}$ and, by Definition~\ref{def:totheright}, $J>I$.
\end{proof}

\subsection{Proof of Theorem~\ref{thm:allofRprecise}}\label{subsec:thmallofR}


\begin{proof}[Proof of Theorem~\ref{thm:allofRprecise}]

First we note that by Proposition \ref{prop:allofQ}, $\omega(\G(\Sigma))$ contains $\mathbb{Q}$. Thus, we only need to show that the image of $\omega$ contains $\mathbb{R}\setminus\mathbb{Q}$.

Fix $\lambda\in \mathbb{R}\setminus\mathbb{Q}$. First suppose that $\Sigma$ is of type (\ref{it:b}). Here, $\Sigma=\Sigma_P\cup S_i$ where $\Sigma_P$ is the surface in Proposition \ref{prop:allofRPlane-discs} and each $S_i$ is homemomorphic to $S$. We can extend the wagon wheel map $\Phi$ with $\omega(\Phi)=\lambda$ defined on $\Sigma_P$ to $\Sigma$. We call this map $\Phi_\Sigma$; it sends $S_i$ to $S_{i+1}$. To see that $\omega(\Phi_\Sigma)=\lambda$ we can consider the same arc $\gamma$ as in the proof of Proposition \ref{prop:allofRPlane-discs}.

Now suppose that $\Sigma$ is of type (\ref{it:a}). Cut $\Sigma$ along the separating circle defining the connected sum. Then we have two pieces, $\Sigma_C$ and $\Sigma_R$,  $\Sigma_C$ is homeomorphic to a disk with a Cantor set and an open disk $B$ removed, and $\Sigma_R$ is the remainder of the surface. 

Up to homeomorphism, we can take $\Sigma_C$ to be the closed disk $D_2$ of radius $2$ around $0$ in $\C$ from which we remove both the Cantor set $C\subset S^1\subset\C$ and the open disk $B$ centered at the origin of radius $1/4$. Up to isotopy the map $g$ with $\omega(g)=\lambda$ defined in Proposition~\ref{prop:allofRCTree} can be chosen to be the identity outside of a small annular neighborhood of $S^1$; hence, it fixes $B$. Now we restrict $g$ to $\Sigma_C$ and extend $g$ to a map $g_\Sigma$ on $\Sigma$ by the identity on $\Sigma_R$.

To see that $\omega(g_\Sigma)=\lambda$ we consider the arc $\gamma$ going straight from 2 to $-2$; this is not contained in $\Sigma_C$ so we replace the straight subarc of $\gamma$ inside of $B$ with a circular arc going around $\partial B$. Using this modified $\gamma$, the argument from Proposition \ref{prop:allofRCTree} again shows that $\omega(g_\Sigma)=\lambda$.
\end{proof}

\section{Realizing any rational number}
\label{section:rationals}

In this section we prove Proposition \ref{prop:allofQ}, which states that
for all $\Sigma$ of infinite type with a compact boundary component $\partial$ and each $q\in \Q$, there exists $g\in\G(\Sigma)$ such that $\fd(g)=q$.

\begin{proof}[Proof of Proposition \ref{prop:allofQ}]
For every $n \geq 2, n \in \mathbb{N}$, we let $D_{n}$ be the disk with $n$ punctures $p_{1},\dots, p_{n}$ and $B_{n}$ its mapping class group, namely the braid group on $n$ strands.
We use that for every rational number $q$ there exists an $n\in\N$ and pure braid $\beta_q\in B_n$ such that $\fd(\beta_q)=q$. We recall that a pure braid is an element that induces the identity on the set of punctures (equivalently the space of ends) of $D_n$. Since, as far as we know, such examples are not provided in the literature, we provide a $\beta_q$ for all $q\in\Q$ in Example~\ref{Ex:purebraidwithprescribedFDTC} below.

We first consider an intermediate surface which we call $D_n'$ obtained from $D_n$ by expanding each puncture $p_{i}$ to a small boundary component $P_{i}$ , and we choose an extension of $\beta_q$ to a homeomorphism $\beta_q'$ of $D_n'$ that fixes all boundary components pointwise. The point is now that, for all $n\in\N$, up to diffeomorphism, each surface $\Sigma$ of infinite type with a chosen compact boundary component $\partial$ is given as the connected sum of the closed disk $\D$ with $n$ surfaces that are not the two-sphere, where the boundary coming from $\D$ is the chosen boundary component. We explain this in the next paragraph.

Without loss of generality, we may and do assume that $\partial \Sigma=\partial$ (otherwise, remove all boundary components except $\partial$, find the connected sum decomposition and put them back as needed). 
If $\Sigma$ has infinite genus, we can view it as a connected sum where one piece is a surface with genus $n-1$ and one boundary component $\partial$ and the other piece is another surface $R$. We therefore can construct $\Sigma$ by gluing in a one-holed torus to each $P_{1}, \ldots, P_{n-1}$ on $D_{n}'$, and gluing in $R$ (after removing an open disc) to $P_{n}$. If instead $\Sigma$ has finite genus, its space of ends must be infinite (since $\Sigma$ is of infinite type). If the set of isolated ends is infinite, we can choose $n-1$ isolated ends and we can view $\Sigma$ as a connect sum of an $n-1$-punctured disk with one boundary component and some other surface $R$. We therefore can construct $\Sigma$ by gluing in a once-punctured disk to each $P_{1}, \ldots, P_{n-1}$, and gluing in $R$ (after removing an open disc) to $P_{n}$. Finally, if the genus is finite and the set of isolated ends is finite, the space of ends must be homeomorphic to a Cantor set union a finite set. In this case, we divide the set of punctures into $n-1$ clopen sets 
each of which is homeomorphic to the Cantor set, and one clopen set homeomorphic to  a Cantor set union a finite set. We therefore can construct $\Sigma$ by gluing in a disk with a Cantor set of punctures removed to each $P_{1}, \ldots, P_{n-1}$, and gluing in whatever remains of $\Sigma$ to $P_{n}$.

In the above paragraph, we constructed $\Sigma$,
 up to a diffeomorphism, from gluing $n$ disjoint surfaces with a special boundary component to $D_n'$. In particular, via this construct, $D_n'$ naturally is a submanifold of $\Sigma$. We have that $\beta_q'$ extends naturally to a map $\Phi$ on $\Sigma$. Moreover, the map fixes all boundary components, so $[\Phi]\in \G(\Sigma)$.

Finally, it remains to observe that $\fd(\Phi)=q$ by construction.
This follows since any essential arc $a$ in $D_n$ that starts on $\partial D_n$, but for concreteness, say an arc that starts and ends on $\partial$ and loops once around the first puncture, has the property that the image $a'$ in $\Sigma$ (under the natural inclusion) satisfies that $\Phi^k(a')$
is the image (under the natural inclusion) of
$\beta_q^k(a)$ .
 \end{proof}

In the calculation below, we freely use basic notions from braid theory (such as the standard Artin generators $\sigma_i$) as e.g.~explained in~\cite{Birman_74_BraidsAMSStudies}.

\begin{Example}\label{Ex:purebraidwithprescribedFDTC}
For a fixed integer $n\geq 3$, we consider the pure braid \[\beta_n\coloneqq\sigma_1^2\sigma_2^2\dots\sigma_{n-1}^2\in B_n.\]
A calculation, provided below, shows that the FDTC of $\beta_n$ is $\frac{1}{n-2}$. Hence, for all pairs of integers $k,l$ with $k\geq 1$, we have that $\beta_\frac{l}{k}\coloneqq \beta_{k+2}^l\in B_{k+2}$ satisfies
$\fd(\beta_\frac{l}{k})= \frac{l}{k}$.

We first show $(n-2)\fd(\beta_n)=\fd(\beta_n^{n-2})\leq 1$. We use that 
$\beta_n^{n-2}=\alpha_l\alpha_r$, where
\begin{align*}
\alpha_l&=(\sigma_1^2\sigma_2^2\dots\sigma_{n-2}^2)(\sigma_1^2\sigma_2^2\dots\sigma_{n-3}^2)\dots(\sigma_1^2\sigma_2^2)(\sigma_1^2)\\
\alpha_r&=(\sigma_{n-1}^2\sigma_{n-2}^2\dots\sigma_{2}^2)(\sigma_{n-1}^2\sigma_{n-2}^2\dots\sigma_{3}^2)\dots(\sigma_{n-1}^2\sigma_{n-2}^2)(\sigma_{n-1}^2).
\end{align*}
Since $\fd(\alpha_l)=\fd(\alpha_r)=0$ (as they are both braids that are given in terms of braid words that miss a standard generator or, put differently, they leave an $I\in\mathfrak{I}$ invariant, namely the one given by an arc that separates off the last, respectively the first, puncture), this yields $\fd(\beta_n^{n-2})\leq \fd(\alpha_l)+\fd(\alpha_r)+1 = 0+0+1$ as desired.

Finally, we check that $(n-2)\fd(\beta_n)=\fd(\beta_n^{n-2})\geq 1$. We notice that by deleting positive generators in the braid word $\beta_n^{n-2}$, we obtain the word $$\sigma_{1}^2\sigma_{2}\dots\sigma_{n-2}\sigma_{n-1}^{2}\sigma_{n-2}\dots\sigma_{2}$$ which is conjugate to $\delta_n\overline{\delta_n}$, where
\[\delta_n\coloneqq\sigma_1\sigma_2\dots\sigma_{n-2}\sigma_{n-1}\quad\text{and}\quad \overline{\delta_n}\coloneqq\sigma_{n-1}\sigma_{n-2}\dots\sigma_{2}\sigma_{1}.\]
The braid $\delta_n\overline{\delta_n}$ is rather well-known to satisfy $\fd(\delta_n\overline{\delta_n})=1$ (see \cite[Section 11.6]{malyutin2005twist}. In conclusion, we have $\fd(\beta_n^{n-2})\geq\fd(\delta_n\overline{\delta_n})=1$, where in the inequality we used that deleting positive generators in braids never increases the FDTC (see \cite[Lemma 5.2]{malyutin2005twist} or Corollary~\ref{cor:insertrvincreasesdtc}).

\end{Example}

\section{Tameness and irrational rotation behavior of the wagon wheel map}\label{sec:tameness}

As discussed in the introduction, we believe the wagon wheel map to be the first instance of irrational rotation behavior beyond the construction of an irrational rotation based on Denjoy's examples. In this section we make this precise. We note that in this section 
we restrict to surfaces that do not have boundary. This choice is made so that we can directly work with the setup for big mapping class groups from~\cite{BFT}.

We first show that wagon wheel maps and variants thereof are tame (in the sense of~\cite{BFT}); see Proposition~\ref{prop:wagonwheelistame}. A self-diffeomorphism $f\colon \Sigma\to \Sigma$ is said to be \emph{tame}, if for all homotopy classes of closed curves $\alpha$ and $\beta$ in $\Sigma$ there exists a constant $C(\alpha,\beta)$ such that the geometric intersection number of $f^n(\alpha)$ and $\beta$ is less than $C(\alpha,\beta)$ for all $n\in\N$~\cite[Lemma~4.1]{BFT}.
For context, we note that ~\cite{BFT}  suggest studying tame maps as an interesting subclass of big mapping class groups on which there is hope to have a Nielsen-Thurston type classification. In fact, the authors of~\cite{BFT} propose as a conjecture that all tame maps ought to be made up of pieces that are translations, periodic maps or maps that have irrational rotation behavior~\cite[Conjecture~B]{BFT}. The term ``irrational rotation behavior'' is defined by example: namely, by considering maps that arise by extending a Denjoy homeomorphism of the circle with irrational rotation number and an invariant Cantor set to a surface of infinite type. In fact, the authors show that if maps are what they call extra tame, then they are made up of translations and periodic maps~\cite[Theorem~A]{BFT}. 

We check that while that wagon wheel maps are tame, they are not extra tame (see Remark~\ref{rem:notextratame}), making them candidates  for ``irrational rotation behavior'' in the context of~\cite[Conjecture~B]{BFT}.
We then go on to propose a definition of irrational rotation behavior by virtue of using the DTC to define a rotation number and check that with this definition, the wagon wheel maps and the irrrational rotations obtained by extending a Denjoy homeomorphism do indeed have irrational rotation behavior; see Propositions ~\ref{prop:irrbehforirrarots} and ~\ref{prop:irrbehforwagonwheels}.

Before doing all of the above, we fix notation and make precise what wagon wheel maps we are considering.
\subsection{The wagon wheel map on surfaces constructed from $\Sigma_P$}\label{subsec:constr}
We let $\Sigma_P$ be the surface from Proposition~\ref{prop:allofRPlane-discs} and
let $\Sigma$ be any surface that is obtained from $\Sigma_P$ as follows: glue a closed disk or a once-punctured closed disk to $\Sigma_P$ along $\partial$ and glue copies of a connected surface $S$ with one boundary component $\partial_S$ that is not a disk to all the other boundary components. Formally, by picking an identification of all boundaries $\partial D_k$ with $\partial S$ and an identification of $\partial$ with $\partial D$, where $D$ is either a closed disk or a closed disk with a puncture, we set
\[\Sigma\coloneqq \Sigma_P\cup D\cup \Z\times S/\sim,\] where the quotient-relation $\sim$ is given by gluing along the above described identifications.
{For example, the result of gluing in a once-punctured closed disk to each boundary component of $\Sigma_P$ or the result of gluing in a once-punctured closed disk along all boundary components except $\partial$ and filling in a disk at $\partial$ are two examples of $\Sigma$'s that arise. }

Note that for every $\lambda\in\R$, the wagon wheel map $\phi\colon\Sigma_P\to\Sigma_P$ extends to $\Sigma$, by extending $\phi$ to be the identity on $\Sigma\setminus \Sigma_P$ up to the obvious permutation, that is, if $[z]\in\Sigma$ then:
\[[z]\mapsto
\left\{\begin{array}{ll}
			[\phi(z)]\in \Sigma & \text{if $z\in\Sigma_P$}
			\\
        	{[(k+1, x)]}\in \Sigma & \text{if $z=(k,x)\in \{k\}\times S$}
        	\\
		\end{array} \right.\]

\subsection{The wagon wheel maps are tame but not extra tame}
 
\begin{prop}\label{prop:wagonwheelistame}For an irrational number $\lambda$, let $\Sigma$ be any of the surfaces as defined in Section~\ref{subsec:constr} and let $\phi$ be the wagon wheel map of rotation $\lambda$ on it, i.e.~the result of extending the wagon wheel map of rotation $\lambda$ from $\Sigma_P$ to $\Sigma$. Then $\phi$ is tame.  
\end{prop}
\begin{proof}
Let $S_i$ denote the surfaces we glue along $\partial D_i$. Let $\alpha$ and $\beta$ be curves in $\Sigma$. We want to show that $i(\phi^n(\alpha),\beta)$ is bounded for all $n>0$, where $i(\cdot,\cdot)$ denotes the geometric intersection number.

Throughout the proof we will refer to notation as in the proof of Proposition~\ref{prop:allofRPlane-discs}, and Figure~\ref{fig:SigmaP}. Let $\rho_0$ be the oriented straight line path along the x-axis starting at $(1,0)$ and ending at $(2.25,0)$ on the boundary of $D_0$. Let $\delta_{\partial}$ be the path starting at $(1,0)$ and going around $\partial$ once around in the clockwise direction. Let $\Gamma_{\partial}=\{\delta_{\partial}\}$, $\Gamma_\rho=\{\phi^k(\rho_0):k\in\mathbb{Z}\}$, and let $\Gamma_i$ be the set of oriented paths contained entirely in $S_i \cup \partial S_{i}$, and $\Gamma_S=\cup_i \Gamma_i$. Finally we let

  $$\Gamma=\Gamma_{\partial}\cup\Gamma_\rho \cup \Gamma_S$$
  
  By definition of the paths in $\Gamma$, the sets $\Gamma_{\partial}$, $\Gamma_\rho$ and $\Gamma_{S} $
 are invariant under $\phi$, and further $\phi(\Gamma_i)=\Gamma_{i+1}$.

The key observation is that any curve, in particular $\alpha$, is freely homotopic to a concatenation of paths in $\Gamma$  (and their inverses) which we also call $\alpha$. Then $\phi^n(\alpha)$ is also a concatenation of the same number of paths in $\Gamma$ for each $n>0$ since $\phi$ acts on $\Gamma$.

Now we consider the intersection of $\beta$ with $\{\phi^n(\alpha)\}$. First we isotope $\beta$ so that it intersects $\Gamma_\rho$ transversely and it lies outside of $B_0(1)\cup A_1$. Hence, $\beta\cap \Gamma_\partial=\emptyset$. Since $\beta$ is a compact curve, it intersects only finitely many $S_i \cup \partial S_{i}$. Hence, for large enough $n$ we have that the set $\beta \cap \phi^n(\alpha)\cap \Gamma_S $ is empty.

Now we compute $\beta \cap \phi^n(\alpha)\cap \Gamma_\rho.$ For all $n>0$ the number of paths in $\phi^n(\alpha)$ which are in $ \Gamma_\rho$ is constant, say $l$, since $\phi$ acts on $\Gamma_{\rho}$. Since $\beta\cap A_1=\emptyset$, the part of $\phi^n(\alpha)\cap \Gamma_\rho$ which might intersect $\beta$ are radial straight lines (there are $l$ of them).   The intersection of any set of $l$ radial lines with $\beta$ is bounded (since $\beta$ is compact). Therefore  $\beta \cap \phi^n(\alpha)\cap \Gamma_\rho$ is bounded in terms of $\beta$ and $l$.

We have shown that $\beta \cap \phi^n(\alpha)\cap \Gamma$ is bounded and therefore so is $\beta\cap \phi^n(\alpha)$. Finally, $|\beta\cap \phi^n(\alpha)|\geq i(\phi^n(\alpha),\beta)$.\end{proof}

\begin{Remark}\label{rem:notextratame}
While all the maps $\phi$ from Proposition~\ref{prop:wagonwheelistame} are tame, they are never {extra tame} as defined in~\cite[Section~4]{BFT}. In particular, they do not fall into the class of maps that can be decomposed into translations pieces and periodic pieces as established in~\cite[Theorem~A]{BFT}. We explain why, in terms of the notion of limit set of a curve as defined in~\cite[Section~4]{BFT} without recalling the definition since we make no use of these concepts outside of this remark.

A tame diffeomorphism $\phi\colon\Sigma\to\Sigma$ is \emph{extra tame} if for all closed curves $\alpha\subset \Sigma$ their limit set is finite. For tame maps this limit set is either a  collection of closed curves or a family of lines, where by a line we mean the image of a proper embedding of $\R$ into $\Sigma$. Note that in~\cite{BFT} there is always a choice of background hyperbolic metric; hence, lines and closed curves can, up to isotopy, be thought of geodesics. In the next paragraphs we describe the limit set of $\phi$ for an explicit choice of $\alpha$: it turns out to be an uncountable set of lines. In particular, $\phi$ is not extra tame.  

We first consider the case when $\Sigma$ arises by gluing a punctured disk to $\partial$.
To be explicit  take $\Sigma$ to be the result of taking $\C\setminus \left(0\cup\bigcup_{k\in\Z} D^\circ_k\right)$ (see Figure~\ref{fig:SigmaP} where rather than removing the interior of $B_0(1)$, one only removes the origin)
and gluing $\Z$ many copies of the same surface to $\partial D^\circ_k$ for all $k\in\Z$. We consider the simple closed curve $\alpha$ given as the boundary of the convex hull (in $\C$) of $B_0(1)$ and $D_0$ (see Figure~\ref{fig:SigmaP} to visualize $\alpha$). Up to isotopy, $\alpha$ is the curve that goes once around $\partial$, then follows the straight path $\rho_0$ from $(1,0)\in\partial$ to $(2.25,0)\in\partial D_0$, then goes once around $\partial D_0$, and closes up by traversing backwards $\rho_0$. By considering that, up to homotopy, $\phi^k(\alpha)$ is the curve that goes once around $\partial$, then follows the straight path $\phi^k(\rho_0)$ that starts at $(1,0)\in\partial$ and ends on $\partial D_k$, then goes once around $\partial D_k$, and closes up by traversing backwards along $\phi^k(\rho_0)$, one can check that the limit set of $\alpha$ (with respect to the correctly chosen hyperbolic metric) consists of all the lines $l_\delta$, where $\delta\in\R\setminus \{k\lambda\mid k\in\Z\}$ and $l_\delta$ is defined as follows. Take the straight ray $r_\delta$ from the origin $0\in\C$ to infinity of angle $\delta$ and define $l_\delta$ as the result of replacing, for all $k\in\Z$, the arc $r_\delta\cap D_k$ with the shortest arc in $\partial D_k$ with the same end points.

We next consider the case when $\Sigma$ arises by gluing a disk to $\partial$.
To be explicit  take $\Sigma$ to be the result of taking $\C\setminus \left(\bigcup_{k\in\Z} D^\circ_k\right)$ (see Figure~\ref{fig:SigmaP} where rather than removing the interior of $B_0(1)$, one leaves it in)
and gluing $\Z$ many copies of the same surface to $\partial D^\circ_k$ for all $k\in\Z$. We consider the simple closed curve $\alpha$ given as the boundary of the convex hull (in $\C$) of $D_0$ and $D_1$ (see Figure~\ref{fig:SigmaP} to visualize $\alpha$). To understand $\phi^n(\alpha)$, it turns out to be easiest to work with a different representative of the same isotopy class: up to homotopy, $\alpha$ is the curve that follows the straight path $\rho_0$ from $(1,0)$ to $(2.25,0)\in\partial D_0$, then goes once around $\partial D_0$, traverses backwards $\rho_0$, then follows $\phi(\rho_0)$, then goes once around $\partial D_1$, and finally closes up at $(1,0)$ by traversing $\phi(\rho_0)$ backwards.

This latter description allows us to readily check that, up to homotopy, $\phi^k(\alpha)$ is the curve that that follows the path $\phi^k(\rho_0)$, then goes once around $\partial D_k$, traverses backwards along $\phi^k(\rho_0)$, then follows $\phi^{k+1}(\rho_0)$, then goes once around $\partial D_{k+1}$, and finally closes up at $(1,0)$ by traversing $\phi^{k+1}(\rho_0)$ backwards. While $\phi^k(\alpha)$ can be isotoped to lie outside of $A_1,$ it gets `stuck' on all $D_i$ with $|i| < n$ whose centers $z_i$ have arguments between the arguments of $z_{n-1}$ and $z_n$. From this description, one can check that the limit set of $\alpha$ up to an isotopy consists of all the lines $l_\delta\cup \{0\}\cup l_{\delta+\lambda}$, where $\delta\in\R\setminus \{k\lambda\mid k\in\Z\}$ and $l_\delta$ is defined as in the last paragraph. In either case, we see that the limit set is not finite.
\end{Remark}

\subsection{Rotation number via the DTC and irrational rotation behavior}
Note that for this subsection we no longer restrict to surfaces without boundary.
Let $\Sigma$ be a surface and $g$ a mapping class of $\Sigma$ with one of the following chosen: a compact boundary component that is fixed by $g$ (in which case we assume $g\in \G(\Sigma,\partial)$, where $\partial$ is the chosen boundary component) or a fixed isolated planar end or a fixed point of $g$. {By \emph{a fixed point} of $g$, we mean the equivalence class $[(f,p)]$ of a choice of pair $(f,p)$, where $p$ is a point in the interior of $\Sigma$ and $f$ is a diffeomorphism in the mapping class $g$ that fixes $p$,} considered up to the following notion of equivalence: $(f,p)$ and $(f',q)$ are equivalent if there is a self-diffeomorphism $h$ of $\Sigma$ such that $h^{-1}\circ f\circ h=f'$ and $h(q)=p$.


{Let $[(f,p)]$ be a fixed point of $g$. Then a choice of representative $(f,p)$} induces a mapping class
$g_{(f,p)}=[f|_{\Sigma\setminus\{p\}}]$ of $\Sigma\setminus\{p\}$,
for which $p$ is naturally identified with a fixed isolated planar end. 
We also note that any $\phi$ with a fixed isolated planar end $e$ provides a $g_e\in \G(\Sigma_e,\partial_e)$, where $\Sigma_e$ is obtained by removing an open once punctured disk from $\Sigma$ with the puncture being $e$ and $\partial_e$ is the new boundary component, that is well-defined up to composing with Dehn twists along $\partial_e$: $g_e$ is defined as the restriction of a representative of $g$ that is the identity on the once punctured disk $\Sigma\setminus\Sigma_e$.

\begin{Definition}\label{Def:rotnr}
Let $g$ be a mapping class of a surface $\Sigma$ with either a chosen boundary component that is fixed, or a fixed isolated planar end, or a fixed point, in all cases denoted by $m$.
We define the rotation number $r(g,m)$ of $g$ at $m$ in terms of the DTC $\omega$ as follows: $r(g,m)\coloneqq$
\[\left\{ \begin{array}{lr}\fd(g)+\Z&\text{if $m$ is a boundary $\partial$ and }\fd\colon \G(\Sigma,\partial)\to\R\\
\fd(g_e)+\Z&\text{if $m$ is an isolated planar end $e$ and }
\fd\colon \G(\Sigma_e,\partial_e)\to\R
\\
\fd((g_{(f,p)})_p)+\Z&\text{if $m$ is a fixed point $[(f,p)]$ and }
\fd\colon \G((\Sigma\setminus\{p\})_p,\partial_p)\to\R
\end{array}\right.\]
We say $g$ has $\emph{irrational rotation behavior}$ at $m$ if $r(g,m)\notin\Q/\Z\subset\R/\Z$.
\end{Definition}

With all of this setup, we easily check that the irrational rotations and wagon wheel maps do indeed exhibit irrational rotation behavior. The proofs of the following two propositions are essentially immediate from Definition \ref{Def:rotnr} and the calculation of the DTC as done in the proof of Theorem \ref{thm:allofRprecise}.

\begin{prop}\label{prop:irrbehforirrarots}Let $\Sigma$ be one of the surfaces from Theorem~\ref{thm:allofRprecise}\eqref{it:a} or one of said surfaces with a closed disk or a punctured closed disk glued to $\partial$. Let $g_r$ be the irrational rotation $r\in\R/\Z\setminus \Q/\Z$ on it, i.e.~the result of extending an irrational rotation with rotation number $r$ to $\Sigma$. Then $r(g_r, m)=r\not\in \Q/\Z$, where $m$ is the boundary $\partial$ (in the case the surface is as in Theorem~\ref{thm:allofRprecise}\eqref{it:a}), a fixed point of $g_r$ inside the disk glued to $\partial$ (in the case we glued a disk to $\partial$), or the puncture of the punctured disk that was glued to $\partial$ (in the case we glued a punctured disk to $\partial$). \qed 
\end{prop}

\begin{prop}\label{prop:irrbehforwagonwheels}Let $\Sigma$ be one of the surfaces from Theorem~\ref{thm:allofRprecise}\eqref{it:b} or one of said surfaces with a closed disk or a punctured closed disk glued to $\partial$ (for example, any of the surfaces from Section~\ref{subsec:constr}). Let $\phi$ be the wagon wheel map with rotation $\lambda\in\R\setminus \Q$ on it, i.e.~the result of extending the wagon wheel map from $\Sigma_P$ to $\Sigma$. Then $r(\phi, m)=\lambda+\Z\not\in \Q/\Z$, where $m$ is the boundary $\partial$ (in the case the surface is as in Theorem~\ref{thm:allofRprecise}\eqref{it:b}), a fixed point of $\phi$ inside the disk glued to $\partial$ (in the case we glued a disk to $\partial$), or the puncture of the punctured disk that was glued to $\partial$ (in the case we glued a punctured disk to $\partial$).\qed 
\end{prop}

\section{Other notions of positivity and total orders on mapping class groups}\label{sec:totalorders}

We briefly discuss other choices of positive elements  $P\subseteq \G(\Sigma)$ that define relations comparable to the ones considered in Section~\ref{sec:MCG}.
In particular, we mention that by considering well-orders on sufficiently large sets of homotopy classes of arcs, one can construct total left-orders on $\G(\Sigma)$.


In Section \ref{sec:MCG}, we found that any non-empty subset $\Gamma\subseteq 
\mathfrak{I}$ determines a relation by defining $P_\Gamma\subseteq \G(\Sigma)$ to be the maps $g$ which send all $I\in\Gamma$ to the right, i.e.~$g(I)\geq I$ 
and by defining $g\preceq_{\Gamma} f$ if $g^{-1}f\in P_\Gamma$. All these relations are comparable; in fact $P_\Gamma\subseteq T P_{\Gamma'}$ for all non-empty subsets $\Gamma, \Gamma'\subseteq \mathfrak{I}$.

We now go towards having not only left-invariant relations but left-invariant total orders. Everything that follows is in analogy to the Dehornoy order for the braid group on $n$ strands, which we will use as an illustrative example throughout the section.
By the well-ordering theorem, we can and do choose a well-order $\lessdot$ on any $\Gamma\subseteq 
\mathfrak{I}$. Recall that a well-order $\lessdot$ is a strict total order where any non-empty subset has a minimal element with respect to $\lessdot$. 
For example, in case of the braid group our surface is the $n$-punctured unit disk with the punctures lined up on the real line, our $\Gamma$ consists of the $n-1$ homotopy classes of arcs that are straight lines starting at $-i$ that cut the $n$-punctured disk into $n$ once punctured discs and we choose the well-ordering to be that an arc is smaller than another if its endpoint has smaller real part. 
 
We can now define
\[P_{\Gamma,\lessdot}\coloneqq \left\{g\in{\G(\Sigma)} \mid \exists I\in \Gamma: g(I)\neq I \text{ and }g(I_{min})\geq I_{min}\right\},\]
where $I_{min}$ is the smallest element with respect to $\lessdot$ in $\Gamma$ with $g(I_{min})\neq I_{min}$. In other words, an element $g$ is in $P_{\Gamma,\lessdot}$ if the first (with respect to $\lessdot$) $I\in\Gamma$ it moves is moved to the right. In the case of the braid group with the above choice of $\Gamma$ and $\lessdot$, $P_{\Gamma,\lessdot}$ is the positive cone of the (strict) Dehornoy order $\prec_{\mathrm{Deh}}$.

\begin{prop}\label{prop_mcgleftorder} If $\Gamma=
\mathfrak{I}$, then $P_{\Gamma,\lessdot}$ determines a total left-invariant order on $\G(\Sigma)$, i.e.~the relation $g\prec h \Leftrightarrow g^{-1}h\in P_{\Gamma,\lessdot}$ is a total order on $\G(\Sigma)$. 
\end{prop}

\begin{proof}
Recall that the choice of $P:=P_{\Gamma,\lessdot}$ defines a left-invariant relation $\prec$ on $\G(\Sigma)$ by $g\prec f$ if $g^{-1}f\in P$. To show that $\prec$ is a strict total order, we need to show that (1) $P\cdot P\subset P$ (transitivity), (2) $P\cap P^{-1}=\emptyset$ (asymmetry) and (3) $P\cup P^{-1}\cup \{[\mathrm{id}]\}=\G(\Sigma)$ (totality). 

To show (2), take $g\in P$ and let $I_{min}\in \Gamma$ be such that $g(I_{min})>I_{min}$ and $g(I)=I$ for all $I\in\Gamma$ with $I\lessdot I_{min}$. Since the relation $\leq$ on $\mathfrak{I}$ is left-invariant under the $\G(\Sigma)$ action, we have $g^{-1}(I_{min})<I_{min}$ and $g^{-1}(I)=I$ for all $I\in\Gamma$ with $I\lessdot I_{min}$; hence, $g^{-1}$ is not in $P$.

To show (1) let $f$ and $g$ be in $P$, we want to show that $g\circ f$ is in $P$. Let $I_{min}$ be the smallest element in $\Gamma$ that $f$ moves (to the right), and $J_{min}$ be the smallest element in $\Gamma$ that $g$ moves (to the right). If $I_{min}=J_{min}$ there is nothing to show; suppose $J_{min}\lessdot I_{min}$. All $K\lessdot I_{min}$ are fixed by $f$, in particular $J_{min}$ is fixed.
Hence, any arc smaller than $J_{min}$ is fixed by $g\circ f$ and $g\circ f(J_{min})>J_{min}$, which means $g\circ f\in P$.  Now suppose that $I_{min}\lessdot J_{min}$. We have $g(f(K))=g(K)=K$ for $K\lessdot I_{min}\lessdot J_{min}$ and (by left-invariance) $g(f(I_{min}))>g(I_{min})=I_{min}$, which means $g\circ f\in P$.

To show (3) we claim that every non-trivial element in $\G(\Sigma)$ must move some element in $\Gamma$; therefore, it is in $P$ or $P^{-1}$. This follows by Alexander's method, which works for surfaces of infinite type; see e.g.~\cite{HernandezMoralesValdez19}. We provide a somewhat detailed argument.
To see the claim, pick $\phi$ in $\G(\Sigma)$ with $\phi(I)=I$ for all $I\in\Gamma$. We argue that $\phi$ is isotopic to the identity. 

First note that since for all boundary components $\partial'$, $\phi$ in particular fixes all $I$ leaving from the chosen basepoint $x_{0} \in \partial$ and terminating on $\partial'$, $\phi$ must fix all boundary components of $\Sigma$ setwise (even pointwise for those boundary components that remain fixed in the definition of $\G(\Sigma)$). Hence, we can choose a representative of $\phi$ that fixes all boundary components pointwise and we may and do assume $\G(\Sigma)=\G(\Sigma,\partial \Sigma)$. 
Removing all boundary components except $\partial$, we find a new surface and a mapping class on it, and we note that $\phi$ is the trivial element if and only if this new mapping class on this new surface is (because $\phi$ can be reconstructed by considering $I$ starting on $\partial$ and ending on a boundary component of ones choice). Therefore, we may and do assume that $\partial\Sigma=\partial$.
Now, we glue to the boundary component $\partial$ of $\Sigma$ a punctured disk; call this new surface $\Sigma'$ and denote by $p$ the puncture replacing the boundary component $\partial$. We extend $\phi$ to a map $\phi'$ on $\Sigma'$ via the identity on the punctured disk. As in \cite[Proof of Lemma 4.6]{schaffer2022automorphisms}, and switching the point of view to thinking of $p$ as a marked point contained in $\Sigma'$ rather than an isolated end, note that, since $\phi'$ fixes the isotopy classes of all loops based at $p$, it must also fix all isotopy classes of simple closed curves on $\Sigma'$; therefore, $\phi'=[\mathrm{id}_{\Sigma'}]$. As the kernel of the inclusion of $\G(\Sigma)$ into the mapping class group of $\Sigma'$ consists of powers of $T$, we know that $\phi$ is a power of $T$, and, since only the trivial power of $T$ fixes the elements of $\mathfrak{I}$, we conclude $\phi=[\mathrm{id}_\Sigma]$.
\end{proof}

\begin{Remark} In \cite{calegari2004circular}, Theorem A, Calegari proves that the mapping class group of the disk with any compact, totally disconnected subset removed (the primary example being the Cantor set) is left-orderable, a case of Proposition~\ref{prop_mcgleftorder}. He constructs Thurston-type orders \cite{Short_OrderingsmappingclassgroupsafterThurston} analogous to those that exist for surfaces of finite type with boundary, and in particular uses the hyperbolic structure of the universal cover. It is understood by experts in big mapping class groups that the same method works to prove that all surfaces of infinite type with at least one compact boundary component are left-orderable. Our proof of Proposition \ref{prop_mcgleftorder} does not use hyperbolic geometry. Readers interested in orders on $\G(\Sigma)$ will also find the action on the circle for $\G(\Sigma)$ of surfaces that fix at least one isolated end constructed by Bavard and Walker in~\cite[Lemma 5.5.1]{bavard2018two} relevant. In fact, their construction implies that $\G(\Sigma)$ of surfaces that fix at least one isolated end are circularly orderable. We give an independent proof of this as it is an immediate corollary of Proposition~\ref{prop_mcgleftorder}.

\end{Remark}
\begin{corollary}[{\cite[Lemma~5.5.1]{bavard2018two}}]\label{cor:cord} The subgroup of every mapping class group given by all elements that fix a given isolated end is circularly orderable. 
\end{corollary}
\begin{proof} Let $\Sigma$ be a surface with a specified end $e$. Let $G$ denote the subgroup of the mapping class group of $\Sigma$ given by all elements that fix the given isolated end $e$.
Consider the central extension $\Z\to \G(\Sigma',\partial)\to G$, where $\Sigma'$ denotes the surface obtained by replacing $e$ by a boundary component, i.e.~$\Sigma'\coloneqq\Sigma\setminus D^{\circ}_\epsilon$ (here $D^{\circ}_\epsilon$ denotes an open punctured disk representative of the end $e$), where the new boundary component is chosen as $\partial$. In particular, $\G(\Sigma',\partial)/\langle T\rangle\cong G$, where $T$ is the positive Dehn twist about $\partial$. Noting that $T$ is cofinal with respect to the total left-order $\prec$ with positive cone $P_{\Gamma,\lessdot}$ (which is a total left-order by Propostion~\ref{prop_mcgleftorder}), we get a cyclic left-order $c$ on $\G(\Sigma',\partial)/\langle T\rangle$ by defining $c(a_1,a_2,a_3)=0$ if the $a_i$ are not pairwise distinct and $c(a_1,a_2,a_3)=\mathrm{sign}(\sigma)$ otherwise, where $\sigma$ is a permutation defined as follows. Let $g_i$ be the unique element in the coset $a_i$ with $1\preceq g_i \prec T$ and $\sigma$ is the unique permutation such that $g_{\sigma(1)}\prec g_{\sigma(2)} \prec g_{\sigma(3)}$; see for instance \cite[Construction 2.2]{ba2021circular} for details.
\end{proof}

While the total left-orders that arise (e.g.~by setting $\Gamma=\mathfrak{I}$) might be the most interesting, all the relations defined by $g\prec f$ if $g^{-1}f\in P_{\Gamma,\lessdot}$, for any non-empty $\Gamma$ and any well-order $\lessdot$, are comparable. Indeed, one can check that each of them is comparable to the relation defined by $\{J\}$, where $J$ is the smallest element in $\Gamma$. Hence, they all can be used to characterize the DTC.
Concretely, we find the following variant of Proposition~\ref{prop:charFDTCforSurfaces}.

\begin{prop}\label{prop:charFDTCforSurfacestotal}
For every non-empty subset $\Gamma$ of $\mathfrak{I}$ and every well-order  $\lessdot$ on $\Gamma$, 
there exists a unique homogeneous quasimorphism $\fd\colon \G(\Sigma)\to \R$  with $\fd(T)=1$ such that
$\fd(P_{\Gamma,\lessdot})$ is bounded below.
Furthermore, $\fd$ is the same for all such non-empty subsets $\Gamma$ and well-orders $\lessdot$ and it satisfies $\fd(P_{\Gamma,\lessdot})\subseteq [0,\infty)$ for all of them. In fact, $\fd$ is the DTC of $\G(\Sigma)$. 
\end{prop}
\begin{proof}
All the relations $P_{\Gamma,\lessdot}$ are comparable to (one and hence all) the relations defined by non-empty $\Gamma\subset \mathfrak{I}$ as considered in Proposition~\ref{prop:charFDTCforSurfaces}. Hence, for all choices of $\Gamma$ and $\lessdot$ we have (by Proposition~\ref{prop:charFDTCforSurfaces}) that $\fd$ is the unique homogeneous quasimorphism with $\fd(T)=1$ such that
$\fd(P_{\Gamma,\lessdot})$ is bounded below. It remains to show that $\fd(P_{\Gamma,\lessdot})\geq 0$.
If $g\in P_{\Gamma,\lessdot}$, then there exist an $I\in \mathfrak{I}$ with $g(I)\geq I$.
Then $g\in P_{I}$; hence, $\fd(g)\in \fd(P_I)\subseteq [0,\infty)$.
\end{proof}
Proposition~\ref{prop:charFDTCforSurfacestotal} is a true generalization of Proposition~\ref{prop:CharPropForFDTC}---Malyutin's result for braids---since, for the choice of $\Gamma$ and $\lessdot$ as described at the beginning of this section, the corresponding relation $\prec$ is the Dehornoy order (under an identification of $\G(\Sigma)$ with $B_n$ that maps $T$ to $\Delta^2$).

 \bibliographystyle{alpha} \bibliography{peterbib}
\end{document}